\documentclass[12pt,reqno]{amsart}
\usepackage{amsmath, amssymb, latexsym, amscd, amsthm,amsfonts,amstext}
\usepackage[mathscr]{eucal}
\usepackage{graphics}
\theoremstyle{plain}

\theoremstyle{remark}

\newtheorem{nx}{\bf Remark}
\theoremstyle{definition}

\renewcommand{\dim}{{\mathrm{dim}}}

\newtheorem{df}{Definition}[section]
\newtheorem{lem}[df]{Lemma}
\newtheorem{cor}[df]{Corollary}

\newtheorem{thr}[df]{Theorem}

\newcommand{\C}{{\mathbb{C}}}

\newcommand{\N}{\mathbb{N}}

\renewcommand{\P}{{\mathbb{P}}}

\newcommand{\R}{{\mathbb{R}}}

\newcommand{\supp}{\mathrm{Supp}\,}

\newcommand{\Z}{\mathbb{Z}}

\numberwithin{equation}{section}

\theoremstyle{definition}
\theoremstyle{plain}

\begin{document}
\date{} 

\title[Normal families of meromorphic mappings ]{Normal families of meromorphic mappings 
of several complex variables for moving hypersurfaces in a complex projective space}

\author{Gerd Dethloff and Do Duc Thai and Pham Nguyen Thu Trang} 

\thanks{The research of the authors is partially supported by a NAFOSTED grant of Vietnam (Grant No. 101.01.38.09).}

\maketitle      
\begin{abstract}
The  main aim of this article is to give sufficient conditions for a family of meromorphic mappings of a domain $D$ in $\C^n$ into $\P^N(\C)$  to be meromorphically normal if they satisfy only some very weak conditions with respect to 
moving hypersurfaces in $\P^N(\C)$, namely that their intersections with these moving hypersurfaces, which may moreover depend on the meromorphic maps,
are in some sense uniform. Our results generalize and
complete previous results in this area, especially the works of Fujimoto \cite{Fu}, Tu \cite{Tu1}, \cite{Tu2},  Tu-Li \cite{Tu-Li}, Mai-Thai-Trang \cite{MTT} and the recent work of Quang-Tan \cite{QT}. 
\end{abstract}

\section{Introduction.}
Classically, a family $\mathcal F$ of holomorphic functions on a domain $D\subset \C$ is said to be (holomorphically) normal
if every sequence in $\mathcal F$ contains a subsequence which converges uniformly on every compact subset of $D$ to a holomorphic map from $D$ into $P^1$.

In 1957 Lehto and Virtanen \cite{LeVi} introduced the concept of normal
meromorphic functions in connection with the study of boundary behaviour 
of meromorphic functions of one complex variable. Since then
normal families of holomorphic maps have been studied intensively, resulting in an extensive
development in the one complex variable context and in generalizations to the
several complex variables setting (see \cite{Za}, \cite{JK1}, \cite{JK2}, \cite{AK} and the
references cited in \cite{Za} and \cite{JK2}).

The first ideas and results on normal families of meromorphic mappings of several complex variables
were introduced by Rutishauser \cite{Rut} and Stoll \cite{S1}.
 
The notion of a meromorphically normal family into the $N$-dimen\-sional complex projective space was introduced by H. Fujimoto \cite{Fu} (see subsection \ref{mero} below  for the definition of these concepts). Also in 
\cite{Fu}, he gave some sufficient conditions for a family of meromorphic mappings of a domain $D$ in $\C^n$ into $\P^N(\C)$ to be meromorphically normal. In 2002, Z. Tu \cite{Tu2} considered meromorphically normal families of meromorphic mappings of a domain $D$ in $\C^n$ into $\P^N(\C)$ for hyperplanes. Generalizing the above results of Fujimoto and Tu, in 2005, Thai-Mai-Trang 
\cite{MTT} gave a sufficient condition for the meromorphic normality of a family of 
meromorphic mappings of a domain $D$ in $\C^n$ into $\P^N(\C)$ for fixed hypersurfaces (see section 2 below for the necessary definitions,
in particular subsection 2.3 for the definition of $D(...)$):

\noindent
{\bf Theorem A.} (\cite[Theorem A]{MTT})\ {\it Let $\mathcal F$ be a family of meromorphic mappings of a domain $D$ in
$\C^n$ into $\P^N(\C)$. Suppose that for each $f\in \mathcal F$, there exist $q\ge 2N+1$
hypersurfaces $H_1(f),H_2(f),...,H_q(f)$  in 
$\P^N(\C)$  with 
$$inf \big\{D(H_1(f),...,H_q(f)) ;f\in{ \mathcal F} \big\} > 0
\text { and }f(D)\not\subset H_i(f)\ (1\le i \le N+1),$$
where $q$ is independent of $f$, but the hypersurfaces $H_i(f)$ may depend on $f$,
such that the following two conditions are satisfied:

i)  For any fixed compact subset $K$ of $D$, 
the $2(n-1)$-dimensional Lebesgue areas of $f^{-1}(H_i(f))\cap K$
$(1\le i \le N+1)$
with counting multiplicities  are bounded above for all $f$ in $\mathcal F$.

ii) There exists a closed subset $S$ of $D$ with $\Lambda^{2n-1}(S)=0$ such that for 
any fixed compact subset $K$ of $D-S$, 
the $2(n-1)$-dimensional Lebesgue areas of $f^{-1}(H_i(f))\cap K$
$(N+2\le i \le q)$
with counting multiplicities  are bounded above for all $f$ in $\mathcal F$.

Then $\mathcal F$ is a meromorphically normal family on $D$.}

Recently, motivated by the investigation of Value Distribution Theory for moving hyperplanes (for example Ru and Stoll  \cite{RS1},  \cite{RS2}, Stoll \cite{S2},  and Thai-Quang  \cite{TQ1},  \cite{TQ2}), the study of the normality of families of 
meromorphic mappings of a domain $D$ in $\C^n$ into $\P^N(\C)$ for moving hyperplanes or  hypersurfaces has started. While a substantial amount of
information has   been   amassed   concerning 
the normality of families of meromorphic mappings for fixed targets through the
years, the present knowledge of this problem for moving targets has
remained extremely meagre. There are only a few such  results in some restricted situations (see  \cite{Tu-Li},  \cite{QT}). For instance, we recall a recent 
result of Quang-Tan  \cite{QT} which is the best result available 
at present and which generalizes Theorem 2.2 of Tu-Li \cite{Tu-Li}:

\noindent
{\bf Theorem B.} (see  \cite[Theorem 1.4]{QT})\ {\it  Let $\mathcal F$ be a family of meromorphic mappings of a domain
$D\subset \C^n$ into $\P^N(\C),$ and let $Q_1,\cdots,Q_q\ (q \geq 2N+1)$ be $q$ moving hypersurfaces
in $\P^N(\C)$  in (weakly) general position such that 

i) For any fixed compact subset $K$ of $D,$ the $2(n-1)$-dimensional
Lebesgue areas of $f^{-1}(Q_j)\cap K \ (1 \leq j \leq N+1)$ counting multiplicities
are uniformly bounded above for all $f$ in $\mathcal F.$

ii) There exists a thin analytic subset $S$ of $D$ such that for any fixed
compact subset $K$ of $D,$ the $2(n-1)$-dimensional Lebesgue areas of
$f^{-1}(Q_j)\cap (K-S)\ (N + 2 \leq j \leq q)$ regardless of multiplicities are
uniformly bounded above for all $f$ in $\mathcal F.$

Then $\mathcal F$ is a meromorphically normal family on $D.$}

We would like to emphasize that, in  Theorem B, the $q$ moving hypersurfaces $Q_1,\cdots,Q_q$ in $\P^N(\C)$ are independent on $f\in \mathcal F$ (i.e. they are common
for all $f\in\mathcal F.$)  Thus, the following question arised na\-turally at this point: {\it Does Theorem A hold for moving hypersurfaces $H_1(f),H_2(f),...,H_q(f)$ which  may depend on $f\in\mathcal F?$}
The main aim of this article is to give an affirmative answer to this question. Namely, we prove the following result which generalizes both Theorem A and Theorem B:

\begin{thr}\label{theorem 1}
{\it Let $\mathcal F$ be a family of meromorphic mappings of a domain $D$ in
$\C^n$ into $\P^N(\C)$. Suppose that for each $f\in \mathcal F$, there exist $q\ge 2N+1$
moving hypersurfaces $H_1(f),H_2(f),...,H_q(f)$  in 
$\P^N(\C)$  
such that the following three conditions are satisfied:

i) For each $1 \leqslant k \leqslant q,$ the coefficients of the homogeneous polynomials $Q_k(f)$ which define the $H_k(f)$  are bounded above uniformly for all $f$ in $\mathcal F$ on compact subsets of $D$, and for any sequence $\{ f^{(p)}\} \subset \mathcal F$, there exists $z \in D$ (which may depend on the sequence) such that 
$$inf_{p \in \N}\big\{D(Q_1(f^{(p)}),...,Q_q(f^{(p)}))(z) \big\} > 0 \mbox \,.$$

ii)  For any fixed compact subset $K$ of $D$, 
the $2(n-1)$-dimensional Lebesgue areas of $f^{-1}(H_i(f))\cap K$
$(1\le i \le N+1)$
counting multiplicities  are bounded above for all $f$ in $\mathcal F$ (in particular $f(D)\not\subset H_i(f)\ (1\le i \le N+1)$).

iii) There exists a closed subset $S$ of $D$ with $\Lambda^{2n-1}(S)=0$ such that for 
any fixed compact subset $K$ of $D-S$, 
the $2(n-1)$-dimensional Lebesgue areas of $f^{-1}(H_i(f))\cap K$
$(N+2\le i \le q)$
ignoring multiplicities  are bounded above for all $f$ in $\mathcal F$.

Then $\mathcal F$ is a meromorphically normal family on $D$.}

\end{thr}

In the special case of a family of holomorphic mappings, we get with the same proof methods:
\begin{thr}\label{theorem 2}
{\it Let $\mathcal F$ be a family of holomorphic mappings of a domain $D$ in
$\C^n$ into $\P^N(\C)$. Suppose that for each $f\in \mathcal F$, there exist $q\ge 2N+1$
moving hypersurfaces $H_1(f),H_2(f),...,H_q(f)$  in 
$\P^N(\C)$  
such that the following three conditions are satisfied:

i) For each $1 \leqslant k \leqslant q,$ the coefficients of the homogeneous polynomials $Q_k(f)$ which define the $H_k(f)$  are bounded above uniformly  for all $f$ in $\mathcal F$ on compact subsets of $D$, and for any sequence $\{ f^{(p)}\} \subset \mathcal F$, there exists $z \in D$ (which may depend on the sequence) such that 
$$inf_{p \in \N}\big\{D(Q_1(f^{(p)}),...,Q_q(f^{(p)}))(z) \big\} > 0 \mbox \,.$$

ii)  $f(D)\cap H_i(f)= \emptyset \ (1 \leqslant i \leqslant N+1)$  for any $f \in F$ .

iii) There exists a closed subset $S$ of $D$ with $\Lambda^{2n-1}(S)=0$ such that for 
any fixed compact subset $K$ of $D-S,$ 
the $2(n-1)$-dimensional Lebesgue areas of $f^{-1}(H_i(f))\cap K\ (N+2\le i \le q)$
ignoring multiplicities  are bounded above for all $f$ in $\mathcal F$.

Then $\mathcal F$ is a holomorphically normal family on $D$.}
\end{thr}

\begin{nx}\ {\it There are several examples in Tu  \cite{Tu2} showing that the conditions  in $i)$,  $ii)$ and $iii)$ in Theorem \ref{theorem 1} and Theorem \ref{theorem 2} cannot be omitted.}
\end{nx}

We also generalise several results of Tu  \cite{Tu1},  \cite{Tu2},  \cite{Tu-Li}
which allow not to take into account at all the components 
of $f^{-1}(H_i(f))$ of high order:

The following theorem generalizes Theorem 2.1 of Tu-Li \cite{Tu-Li} from the case of moving hyperplanes which are independant of $f$ to moving hypersurfaces which may depend on $f$ (in fact observe that for moving
hyperplanes the condition $H_1, \cdots, H_q$ in $\widetilde{\mathcal S}\big(\{T_i\}_{i = 0}^N\big)$ is satisfied by taking $T_0,...,T_N$ any (fixed 
or moving)
$N+1$ hyperplanes in general position). 

\begin{thr}\label{theorem 4}
{\it Let $\mathcal F$ be a family of holomorphic mappings of a domain $D$ in $\C^n$ into $\P^N\big(\C\big).$ Let $q\geqslant 2N + 1$ be a positive integer. Let  $m_1,\cdots,m_q$ be positive intergers or 
$\infty$ such that
$$\sum_{j=1}^{q}\bigg(1 - \dfrac{N}{m_j}\bigg) > N+1 .$$
Suppose that for each $f \in \mathcal F$, there exist  $N + 1$ moving hypersurfaces 
$T_0\big(f\big),\cdots,T_N\big(f\big)$ in $\P^N\big(\C\big)$ of common degree and there exist $q$ moving hypersurfaces $H_1\big(f\big), \cdots, H_q\big(f\big)$ in $\widetilde{\mathcal S}\big(\{T_i\big(f\big)\}_{i = 0}^N\big)$ such that the following conditions are satisfied:

i) For each $0 \leqslant i \leqslant N,$ the coefficients of the homogeneous polynomials $P_i(f)$ which define the $T_i(f)$  are bounded above uniformly 
 for all $f$ in $\mathcal F$ on compact subsets of $D$, and for all 
$1 \leqslant j \leqslant q,$
the coefficients $b_{ij}(f)$ of the linear combinations of the 
$P_i(f)$, $i=0,...,N$ which define the
homogeneous polynomials $Q_j(f)$ which define the $H_j(f)$  are bounded above uniformly   for all $f$ in $\mathcal F$ on compact subsets of $D$, and
for any fixed $z \in D$,
$$\inf \left\{D\big(Q_1\big(f\big), \cdots, Q_q\big(f\big)\big)(z):\  f \in \mathcal F\right\} > 0.$$

ii) $f$ intersects $H_j\big(f\big) $ with multiplicity at least $m_j $ for each $1 \leq j\leq q$  (see subsection 2.6 for the necessary definitions).  

Then $\mathcal F$ is a holomorphically normal family on $D$.}
\end{thr}

The following theorem generalizes Theorem 1 of Tu \cite{Tu2} from the case of fixed hyperplanes  to moving hypersurfaces  (in fact observe that for
hyperplanes the condition $H_1(f), \cdots, H_q(f)$ in $\widetilde{\mathcal S}\big(\{T_i(f)\}_{i = 0}^N\big)$ is satisfied by taking $T_0(f),...,T_N(f)$ any 
$N+1$ hyperplanes in general position).

\begin{thr}\label{theorem 5}
{\it  Let $\mathcal F$ be a family of meromorphic mappings of a domain $D$ in $\C^n$ into $\P^N\big(\C\big).$ Let $q\geqslant 2N + 1$ be a positive integer. Suppose that for each $f \in \mathcal F$, there exist  $N + 1$ moving hypersurfaces 
$T_0\big(f\big),\cdots,T_N\big(f\big)$ in $\P^N\big(\C\big)$ of common degree and there exist $q$ moving hypersurfaces $H_1\big(f\big), \cdots, H_q\big(f\big)$ in $\widetilde{\mathcal S}\big(\{T_i\big(f\big)\}_{i = 0}^N\big)$ such that the following conditions are satisfied:

i) For each $0 \leqslant i \leqslant N,$ the coefficients of the homogeneous polynomials $P_i(f)$ which define the $T_i(f)$  are bounded above uniformly  
 for all $f$ in $\mathcal F$ on compact subsets of $D$, and for all 
$1 \leqslant j \leqslant q,$
the coefficients $b_{ij}(f)$ of the linear combinations of the 
$P_i(f)$, $i=0,...,N$ which define the
homogeneous polynomials $Q_j(f)$ which define the $H_j(f)$  are bounded above uniformly   for all $f$ in $\mathcal F$ on compact subsets of $D$, 
and for any sequence $\{ f^{(p)}\} \subset \mathcal F$, there exists $z \in D$ (which may depend on the sequence) such that 
$$inf_{p \in \N}\big\{D(Q_1(f^{(p)}),...,Q_q(f^{(p)}))(z) \big\} > 0 \mbox \,.$$

ii)  For any fixed compact $K$ of $D,$ the $2(n-1)$-dimensional Lebesgue areas of $f^{-1}\big(H_k(f)\big) \cap K \ (1\leq k \leq N+1)$ counting multiplicities  are bounded above for all $f \in \mathcal F$  (in particular $f\big(D\big) \not\subset  H_k\big(f\big)\ (1\leq k\leq N+1)$).

iii)  There exists a closed subset $S$ of $D$ with $\Lambda^{2n-1}(S)=0$ 
 such that for any fixed compact subset $K$ of $D - S,$  the $2(n-1)$-dimensional  Lebesgue areas of
\begin{equation*}
\left\{z \in \supp \nu\big(f, H_k(f)\big)\big|\nu\big(f, H_k(f)\big)(z) < m_k \right\} \cap K \ (N+2\leq k\leq q)
\end{equation*}
ignoring multiplicities for all $f \in \mathcal F$ are bounded above, where $\{m_k\}_{k = N+2}^q$ are fixed positive intergers or $\infty$ with
\begin{equation*}
\sum_{k =N+2}^q\dfrac{1}{m_k}< \dfrac{q - \big(N+1\big)}{N}.
\end{equation*}

Then $\mathcal F$ is a meromorphically normal family on $D$.}
\end{thr}

 The following theorem generalizes Theorem 1 of Tu \cite{Tu1} from the case of fixed hyperplanes  to moving hypersurfaces.

\begin{thr}\label{theorem 6}
{\it Let $\mathcal F$ be a family of holomorphic mappings of a domain $D$ in $\C^n$ into $\P^N\big(\C\big).$ Let $q\geqslant 2N + 1$ be a positive integer. Suppose that for each $f \in \mathcal F$, there exist  $N + 1$ moving hypersurfaces 
$T_0\big(f\big),\cdots,T_N\big(f\big)$ in $\P^N\big(\C\big)$ of common degree and there exist $q$ moving hypersurfaces $H_1\big(f\big), \cdots, H_q\big(f\big)$ in $\widetilde{\mathcal S}\big(\{T_i\big(f\big)\}_{i = 0}^N\big)$ such that the following conditions are satisfied:

i) For each $0 \leqslant i \leqslant N,$ the coefficients of the homogeneous polynomials $P_i(f)$ which define the $T_i(f)$  are bounded above uniformly 
 for all $f$ in $\mathcal F$   
on compact subsets of $D$, and for all 
$1 \leqslant j \leqslant q,$
the coefficients $b_{ij}(f)$ of the linear combinations of the 
$P_i(f)$, $i=0,...,N$ which define the
homogeneous polynomials $Q_j(f)$ which define the $H_j(f)$  are bounded above uniformly   for all $f$ in $\mathcal F$ on compact subsets of $D$, 
and for any sequence $\{ f^{(p)}\} \subset \mathcal F$, there exists $z \in D$ (which may depend on the sequence) such that 
$$inf_{p \in \N}\big\{D(Q_1(f^{(p)}),...,Q_q(f^{(p)}))(z) \big\} > 0 \mbox \,.$$

ii)   $f(D)\cap H_i(f)= \emptyset \ (1 \leqslant i \leqslant N+1)$  for any $f \in \mathcal F$.

iii) 
There exists a closed subset $S$ of $D$ with $\Lambda^{2n-1}(S)=0$ 
 such that for any fixed compact subset $K$ of $D - S,$  the $2(n-1)$-dimensional  Lebesgue areas of
$$\{z \in \supp \nu\big(f, H_k(f)\big)\big|\nu\big(f, H_k(f)\big)(z) < m_k \} \cap K \ (N+2\leq k\leq q) $$
ignoring multiplicities for all $f$ in $\mathcal F$ are bounded above, where $\{m_k\}_{k = N+2}^q$ are fixed positive intergers and may be $\infty$ with
\begin{equation*}
\sum_{k =N+2}^q\dfrac{1}{m_k}< \dfrac{q - \big(N+1\big)}{N}.
\end{equation*}

Then $\mathcal F$ is a holomorphically normal family on $D$.}
\end{thr}

Let us finally give some comments on our proof methods:

The proofs of Theorem \ref{theorem 1} and Theorem 
\ref{theorem 2} are obtained by generalizing ideas, which have been
used by Thai-Mai-Trang \cite{MTT} to prove Theorem A, to moving targets, which presents several highly non-trivial technical difficulties. Among others,
for a sequence of moving targets $H(f^{(p)})$ which at the same time may depend of the meromorphic maps $f^{(p)} : D \rightarrow \P^N\big(\C\big)$, obtaining a subsequence which converges locally uniformly on $D$ is much more difficult than for fixed
targets (among others we cannot normalize the coefficients to have norm equal to 1 everywhere like for fixed targets). This is obtained in
 Lemma \ref{lemma 5n}, after having proved in  Lemma \ref{lemma 4} that  the condition $D(Q_1,...,Q_q) > \delta > 0$ forces a uniform bound, only in terms of $\delta$, on the degrees of the $Q_i$, $1 \leq i \leq q$
 (in fact the latter result fixes also a gap in \cite{MTT} even for the case of fixed targets). 

 The proofs of Theorem \ref{theorem 4}, Theorem \ref{theorem 5} and Theorem \ref{theorem 6} are obtained by combining methods used by
 Tu \cite{Tu1}, \cite{Tu2} and Tu-Li \cite{Tu-Li} with the methods which we developed to prove our first two theorems. 
 However, in order to apply the technics which Tu and Tu-Li used for the case of hyperplanes, we still need that for every
 meromorphic map $f^{(p)} : D \rightarrow \P^N\big(\C\big)$, the 
 $Q_1(f^{(p)}),...,Q_q(f^{(p)})$ are still in a linear system given by 
 $N+1$ such maps
 $P_0(f^{(p)}),..., P_N(f^{(p)})$. The Lemmas \ref{lemma 11} to 
 Lemma \ref{lemma 5nn} adapt our technics to this situation
  (for example Lemma 
 \ref{lemma 5nn} is an adaptation of our Lemma \ref{lemma 5n})

\section{Basic notions.}

\subsection{Meromorphic mappings.} Let A be a non-empty open subset of a domain $D$ in $\C^n$ 
such that $S=D-A$ is an analytic set in $D$. Let $f:A\to \P^N(\C)$
be a holomorphic mapping. Let $U$ be a non-empty connected open subset of $D$.
A holomorphic mapping $\tilde f\not \equiv 0$
from $U$ into $\C^{N+1}$ is said to be a representation of $f$ on $U$
 if $f(z)=\rho(\tilde f(z))$ for all $z\in U\cap A-\tilde f^{-1}(0)$, 
where $\rho:\C^{N+1}-\{0\} \to \P^N(\C)$ is the canonical projection.
A holomorphic mapping $f:A\to \P^N(\C)$ is said to be a meromorphic mapping
from $D$ into $\P^N(\C)$ if for each $z\in D$, there exists 
a representation of $f$ on some neighborhood of $z$ in $D$.

\subsection{Admissible representations. } Let $f$ be a meromorphic mapping
of a domain $D$ in $\C^n$ into $\P^N(\C)$.
Then for any $a\in D$, $f$ always has an admissible representation 
$\tilde f(z)=(f_0(z),f_1(z),\cdots,f_N(z)) $
on some neighborhood $U$ of $a$ in $D,$ which means that each $f_i(z)$ is a holomorphic function on $U$ and 
$f(z)=(f_0(z):f_1(z):\cdots:f_N(z)) $ outside the analytic set $I(f):= \{ z \in U : f_0(z) = f_1(z) =...= f_N(z) = 0\}$ of
codimension $\ge 2.$ 

\subsection{Moving hypersurfaces in general position. } Let $D$ be a domain in $\C^n$. Denote by $\mathcal H_D$ the ring of all holomorphic functions on $D$, and 
$\widetilde{\mathcal H}_D [\omega_0, \cdots, \omega_N]$
the set of all homogeneous polynomials $Q \in \mathcal H_D[\omega_0, \cdots, \omega_N]$ such that the coefficients of  $Q$ are not all identically zero. Each element of $\widetilde{\mathcal H}_D [\omega_0, \cdots, \omega_N]$ is said to be a moving hypersurface in  $\P^N(\C).$  

Let $Q$ be a moving hypersurface of degree  $d \geqslant 1$. Denote by $Q(z)$ the homogeneous polynomial over $\C^{N+1}$ obtained by evaluating   the coefficients of $Q$ in a specific point $z \in D$.  We remark that for generic $z \in D$ this is a non-zero homogenous polynomial with coefficients in $\C$.
The hypersurface $H$ given by $H(z):=\{w\in \C^{N+1}: Q(z)(w)=0\}\ ($for generic $z\in D)$ is also called to be a moving hypersuface in  $\P^N(\C)$ which is defined by $Q.$  In this article, we identify $Q$ with $H$ if no confusion arises.

We say that moving hypersurfaces $\{Q_j\}_{j=1}^q $ of degree $d_j$ $\ (q \geqslant N+1)$  in $\P^N(\C)$
are located in (weakly) general position if there exists  $z \in D$ such that for any  $1 \leqslant j_0 < \cdots < j_N \leqslant q$, the system of equations 
\begin{equation*}
\left \{ \begin{array}{ll}
Q_{j_i} (z) \big(\omega_0, \cdots, \omega_N\big) = 0\\
0 \leqslant i \leqslant N
\end{array} \right.
\end{equation*}
has only the trivial solution $\omega = \big(0, \cdots, 0\big)$ in $\C^{N+1}$. This is equivalent to

$$D(Q_1,...,Q_q)(z):=\prod_{1 \leq j_0 < \cdots < j_N \leq q} \inf_{||\omega|| = 1} \bigg(\big|Q_{j_0}(z)(\omega)\big|^2 + \cdots + \big|Q_{j_N}(z)(\omega)\big|^2\bigg) > 0,$$
where $Q_j(z)(\omega)=\sum_{| I | = d_j} a_{jI}(z).\omega^I$ and $||\omega|| = \big(\sum |\omega_j|^2\big)^{1/2}$.

\subsection{Divisors. }Let $D$ be a domain in $\C^n$ and $f$ a non-identically zero holomorphic function 
on $D$. For a point $a=(a_1,a_2,...,a_n)\in D$ we expand $f$ as a compactly 
convergent series
$$f(u_1+a_1,.....,u_n+a_n)=\sum_{m=0}^{\infty}P_m(u_1,...,u_n)$$
on a neighborhood of $a$, where $P_m$ is either identically zero 
or a homogeneous polynomial of degree $m$. The number 
$$\nu_{f}(a):= \min  \{m; P_m(u)\not\equiv 0\}$$
is said to be the zero multiplicity of $f$ at $a$.
By definition, a divisor on $D$ is an integer-valued function $\nu$
on $D$ such that for every $a\in D$ there are holomorphic
 functions $g(z)(\not\equiv 0)$ and $h(z)(\not\equiv 0)$
on a neighborhood $U$ of $a$
with $\nu(z)=\nu_g(z)-\nu_h(z)$ on $U$.
We define the support of the divisor $\nu$
on $D$ by 
$$\supp \  \nu:=\overline {\{z\in D:\nu(z)\ne 0\}}.$$
We denote ${\mathcal D^+}(D) =\{\nu:$ a non-negative divisor on $D\}$.

Let $f$ be a  meromorphic mapping from a domain  $D$ into $\P^N\big(\C\big)$. For each homogeneous polynomial $Q \in \widetilde{\mathcal H}_D [\omega_0, \cdots, \omega_N]$, we define the divisor $\nu\big(f, Q\big)$ on $D$ as follows: For each $a \in D$, let $\widetilde{f} = \big(f_0,\cdots, f_N\big)$ be an admissible representation of  $f$ in a neighborhood $U$ of $a$. Then we put
\begin{equation*}
\nu\big(f, Q\big)(a) := \nu_{Q(\tilde f)}(a),
\end{equation*}
where $Q(\tilde f) := Q\big(f_0, \cdots, f_N\big).$ 

Let $H$ be a moving hypersurface which is defined by the homogeneous polynomial $Q \in \widetilde{\mathcal H}_D[\omega_0, \cdots, \omega_N]$, and $f$ be a meromorphic mapping of  $D$ into $\P^N\big(\C\big).$
As above we define the divisor
$\nu(f,H) (z):=\nu\big(f, Q\big)(z).$
Obviously, $\supp \nu(f,H)$
is either an empty set or a pure $(n-1)- $dimensional analytic set in $D$ if
$f(D)\not\subset H$ (i.e., $Q(\tilde f)\not\equiv 0$ on $U$).
We define $\nu(f, H)=\infty$ on $D$ 
and $\supp \nu(f,H)=D$
if $f(D)\subset H$.
Sometimes we identify $f^{-1}(H)$ 
with the divisor $\nu(f, H)$ on $D$.
 We can rewrite $\nu(f, H)$ 
as the formal sum
$\nu(f, H)=\sum\limits_{i\in I} n_iX_i$,
where $X_i$
are the irreducible components of $\supp \nu(f, H)$ and
$n_i$ are the constants
$\nu(f, H)(z)$ on $X_i\cap Reg(\supp \nu(f,H)),$ where
$Reg(\ )$ denotes the set of all the regular points.

We say that the meromorphic mapping $f$ intersects $H$ with multiplicity at least $m$ on $D$ if 
$\nu(f, H)(z)\ge m$ for all $z\in \supp  \nu(f, H)$
and in particular that $f$ intersects $H$ with multiplicity 
$\infty$ on $D$ if 
$f(D)\subset H$ or $f(D)\cap H=\emptyset$.

\subsection{Meromorphically normal families. }\label{mero}\ Let $D$ be a domain in $\C^n.$

\noindent
i)\ (See \cite{AK})  Let $\mathcal F$ be a family of holomorphic mappings of $D$  into a compact complex manifold $M$.
$\mathcal F$ is said to be {\it a (holomorphically) normal family on $D$} if any sequence in $\mathcal F$ contains
 a subsequence which converges uniformly on compact
subsets of $D$ to a holomorphic mapping of $D$ into $M$.

\noindent
ii)\ (See \cite{Fu}) A sequence $\{f^{(p)}\}$ of meromorphic mappings from  $D$ into $\P^N(\C)$ is said {\it to converge meromorphically  on $D$ to a 
meromorphic mapping $f$}
if and only if, for any $z\in D$, each $f^{(p)}$ has an admissible representation
$$\tilde f^{(p)}=(f_0^{(p)}:f_1^{(p)}:...:f_N^{(p)})$$
on some fixed neighborhood $U$ of $z$ such that 
$\{f^{(p)}_i\}_{p=1}^{\infty}$
converges uniformly on compact
subsets of $U$ to a holomorphic function $f_i$  $(0\le i \le N)$ on $U$ 
with the property that 
$\tilde f =(f_0:f_1:...:f_N)$
is a representation of $f$ on $U$ (not necessarily an admissible one ! ).

\noindent
iii)\ (See \cite{Fu}) Let $\mathcal F$ be a family of meromorphic mappings of $D$ into $\P^N(\C)$.
$\mathcal F$ is said to be {\it a meromorphically normal family on $D$} if any sequence in $\mathcal F$ 
has a meromorphically convergent subsequence on $D$.

\noindent
iv)\ (See \cite{S1})   Let $\{\nu_i\}$ be a sequence of non-negative divisors on $D$. It is said to converge to a non-negative divisor $\nu$ on $D$ if and only if 
 any $a\in D$ has a neighborhood $U$ such that there exist
 holomorphic functions $h_i (\not\equiv 0)$ and $h(\not\equiv 0)$
on $U$ such that  $\nu_i=\nu_{h_i}$, $\nu=\nu_h$ and $\{h_i\}$
converges compactly to $h$ on $U$.

\noindent 
v)\ (See \cite{S1}) 
Let $\{A_i\}$ be a sequence of closed subsets of $D$. It is said to converge to a closed subset $A$ of $D$ if and only if 
 $A$ coincides with the set of all $z$ such that every neighborhood $U$ of $z$ intersects $A_i$ for all but finitely many $i$ and, simultaneously, with the set of all $z$ such that every $U$ intersects  $A_i$ for  infinitely many $i$.

\subsection{Other notations.} Let $P_0, \cdots, P_N$ be $N + 1$ homogeneous polynomials of common degree in $\C[\omega_0, \cdots, \omega_N].$ Denote by $\mathcal{S}\big(\{P_i\}_{i=0}^N\big)$ the set of all homogeneous polynomials $Q = \sum_{i=0}^N\limits b_i P_i \; (b_i \in \C).$

Let  $\{Q_j:= \sum_{i = 0}^N\limits b_{ji} P_i\}_{j=1}^q$ be $q \; (q \geqslant N+1)$ homogeneous polynomials in $\mathcal{S}\big(\{P_i\}_{i=0}^N\big).$ 
We say that   $\{Q_j\}_{j=1}^q$ are  located in general position in $\mathcal{S}\big(\{P_i\}_{i=0}^N\big)$ if 
\begin{equation*}
\forall \ 1 \leqslant j_0 < \cdots < j_N \leqslant q, \; \det \big(b_{j_k i}\big)_{0 \leqslant k, i \leqslant N} \not= 0.
\end{equation*}

 Let  $T_0, \cdots, T_N$ be hypersurfaces in $\P^N\big(\C\big)$ of common degree, where $T_i$ is defined by  the (not zero) polynomial $P_i \ (0 \leqslant i\leqslant N).$ Denote by 
 $\widetilde{\mathcal{S}}\big(\{T_i\}_{i=0}^N\big)$ the set of all hypersurfaces in $\P^N\big(\C\big)$ which are defined by $Q \in \mathcal{S}\big(\{P_i\}_{i=0}^N\big)$ with $Q$ not zero.

Let $P_0, \cdots, P_N$ be $N + 1$ homogeneous polynomials of common degree in $\widetilde{\mathcal H}_D [\omega_0, \cdots, \omega_N]$. Denote by $\widetilde{\mathcal{S}}\big(\{P_i\}_{i=0}^N\big)$ the set of all homogeneous not identically zero polynomials $Q = \sum_{i=0}^N\limits b_i P_i \; (b_i \in \mathcal H_D).$

Let  $T_0, \cdots, T_N$ be moving hypersurfaces in $\P^N\big(\C\big)$ of common degree, where $T_i$ is defined by  the (not identically zero) polynomial $P_i \ (0 \leqslant i\leqslant N).$ Denote by $\widetilde{\mathcal{S}}\big(\{T_i\}_{i=0}^N\big)$ the set of all moving hypersurfaces in $\P^N\big(\C\big)$ which are defined by $Q \in \widetilde{\mathcal{S}}\big(\{P_i\}_{i=0}^N\big).$

Denote by $Hol(X,Y)$ the set of all holomorphic mappings from a complex space $X$ to a complex space $Y.$

For each $x\in \C^n$ and $R>0$, we set $B(x,R)= \{z\in \C^n : ||z-x||<R\}$ and 
$B(R)=B(0,R).$

 Let $\Lambda^d(S)$ denote the real $d$-dimensional Hausdorff measure of $S\subset \C^n.$ For a formal $\Z$-linear
combination $X= \sum_{i\in I} n_iX_i$ of analytic subsets $X_i \subset \C^n$ and for a subset $E\subset \C^n,$ we 
call $\sum_{i\in I}\Lambda^d(X_i\cap E)$ \ (resp. $\sum_{i\in I}n_i\Lambda^d(X_i\cap E))$, the $d$-dimensional
Lebesgue area of $X\cap E$  {\it ignoring multiplicities} (resp. {\it with counting multiplicities}).

Finally we list some facts on Hausdorff measures for later use which
can for example be found in the book of Chirka \cite{C}.
\begin{lem}\label{HD} (\cite[p.351f. and p.299f.]{C})
Let $B:=B(x,R) \subset \C^n$ and $S \subset B$ a closed subset with
$\Lambda^{2n-1}(S)=0$. Let $h : B \rightarrow \C$ be a holomorphic function, $
h \not\equiv 0$,  and
$S_1:=S \cup \{h=0\}$. Let $z_0 \in B$. Then we have :\\
\indent a) For almost every complex line $l:= \{z_0 +z \cdot u, \: z \in \C \}$
passing through $z_0$, we have $\Lambda^1(S_1 \cap l ) =0$.\\
\indent b) Let $r_0 := {\rm dist} (z_0,  \partial B)$ and, for every $r \in ]0, r_0[$, $\mathcal C_r:= \{ z_0 +r \cdot e^{i\theta}\cdot u : \theta \in [0, 2\pi]\ \} \subset l$. 
If $\Lambda^1(S_1 \cap l ) =0$ then the subset  of the $r$ such that $\mathcal C_r \cap S_1 \not= \emptyset$ is nowhere dense in the interval $]0, r_0[$.\\

\end{lem}

\begin {cor} \label{HDC} With the notations of Lemma \ref{HD} the set $B \setminus S_1$ is pathwise connected.
\end{cor}
\noindent {\it Proof.} This is an immediate consequence of Lemma \ref{HD} and of the fact that $B \setminus S_1 \subset B$ is an open subset. \qed

 \section{ Lemmas.}

\begin{lem}\label{lemma 2} (\cite[Theorem 2.24]{S1})
{\it A sequence $\{\nu_i\}$ of non-negative divisors on a 
domain $D$ in $\C^n$ is normal in the sense of the convergence of divisors on $D$
 if and only if 
the $2(n-1)$-dimensional Lebesgue areas of $\nu_i\cap E$
$(i \ge 1)$
with counting multiplicities are bounded above
for any fixed compact set $E$ of $D$.}
\end{lem}

\begin{lem}\label{lemma 1} (\cite[Theorem 4.10]{S1})
{\it If a sequence $\{\nu_i\}$ converges to $\nu$ in  ${\mathcal D^+}(B(R))$, then $\{supp \  \nu_i\}$ converges to $supp \  \nu$ (in the sense of closed subsets of $D$)}.
\end{lem}

\begin{lem}\label{lemma 2a} (\cite[Proposition 4.12]{S1})
{\it Let $\{N_i\}$ be a sequence of pure $(n-1)$-dimensional analytic subsets of a domain $D$ in $\C^n$. If the $2(n-1)$-dimensional Lebesgue areas of $ N_i \cap K $ ignoring multiplicities $(i = 1, 2, ...)$ are bounded above for any fixed compact subset $K$ of $D$, then $\{N_i\}$ is normal in the sense of the convergence of closed subsets in $D$.}
\end{lem}

\begin{lem}\label{lemma 1a} (\cite[Proposition 4.11]{S1})
{\it Let $\{N_i\}$ be a sequence of pure $(n-1)$-dimensional analytic subsets of a domain $D$ in  $\C^n$. Assume that the $2(n-1)$-dimensional Lebesgue areas of $N_i \cap K$ ignoring multiplicities $(i = 1, 2, \cdots)$ are bounded above for any fixed compact subset $K$ of $D$ and $\{N_i\}$ converges to $N$ as a sequence of closed subsets of $D$. Then $N$ is either empty or a pure $(n-1)$-dimensional analytic subset of $D$.}
\end{lem}

\begin{lem}\label{lemma 4} 
{\it Let natural numbers $N$ and $q \geqslant N+1$ be fixed. Then for each $\delta > 0,$  there exists $M(\delta)= M(\delta, N,q) > 0$ such that the following is satisfied:

For any homogeneous polynomials $Q_1, \cdots, Q_q $  on $\C^{N+1}$ with complex coefficients with norms bounded above by $1$ such that $D\big(Q_1, \cdots, Q_q\big) > \delta,$
we have $\max\{\deg Q_1, \cdots, \deg Q_q\} < M(\delta).$}
\end{lem}

\begin{proof} 
First of all, we make the  three  following remarks.

\noindent
$i)$\, Let $Q(\omega)$ be a homogeneous polynomial on $\C^{N+1}$ such that
$$Q(\omega) = \sum_{|\alpha|=d} \limits a_\alpha \omega^\alpha,$$ 
where $ |a_\alpha| \leq 1.$ Then
$$|Q(\omega)|\leq\sum_{|\alpha|=d}\limits |\omega_0|^{\alpha_0}\cdots |\omega_N|^{\alpha_N}\leq(d+1)^{N+1}r^{d},$$
when $|\omega_k|\leq r \ \forall \ 0\leqslant k\leqslant N.$

We set
$$M_0=\sup_{d\in \mathbb Z^+}(d+1)^{N+1}\left (\dfrac{1}{\sqrt {N+1}}\right )^{d}.$$
Since $\lim\limits_{d\longrightarrow +\infty}(d+1)^{N+1}\left (\dfrac{1}{\sqrt {N+1}}\right )^{d}=0$, it implies that
$M_0<+\infty.$
 
\noindent
$ii)$\, Let $Q_0,\cdots,Q_N$ be homogeneous polynomials on $\C^{N+1}$ such that the norms of their complex coefficients 
are bounded above by $1$ and $D(Q_0,\cdots,Q_N)>0$. We choose $\omega^{(0)}=(1/\sqrt{N+1},\cdots, 1/\sqrt{N+1})\in\C^{N+1}.$ Then $||\omega^{(0)}||=1$. By (i), we have
\begin{center}
$D(Q_0,\cdots,Q_N)\leq (N+1)M_0^{2}<+\infty .$
\end{center}

\noindent
$iii)$  Since $\lim_{x \rightarrow \infty}\limits (1-\frac{1}{x})^x=\frac{1}{e}$, we have $\frac{e(1-\frac{1}{x})^x}{2} < 1$ for $x$ big enough. Therefore, $$\lim_{x \rightarrow \infty}\limits (x^2+1)^{N+1}(1-\frac{1}{x})^{x^2}=\lim_{x \rightarrow \infty}\limits \left( \frac{e(1-\frac{1}{x})^x}{2} \right)^{x}\frac{(x^2+1)^{N+1}}{(\frac{e}{2})^{x}}=0.$$
We now come back to the proof of Lemma \ref{lemma 4}, and we consider the following two cases. \\
{\bf Case 1:} \,$q=N+1$.\\
Assume that such a constant $M(\delta)=M(\delta, N,N+1)$ does not exist.  Then there exist homogeneous polynomials $Q^{(j)}_0,..,Q^{(j)}_N$ $(j\geqslant 1)$ with coefficients being bounded above by $1$ such that
\begin{center}
$\inf \left\{D(Q^{(j)}_0,..,Q^{(j)}_N):\ j\geq1\right\}>\delta>0,$\\
$\lim_{j\rightarrow\infty}\limits\left(\max \left\{\deg Q^{(j)}_0,\cdots,\deg Q^{(j)}_N\right\}\right)=\infty$.\\
\end{center}
Without loss of generality we may assume that
\begin{center}
$\deg Q^{(j)}_i=d_i\ \forall \ 0\leq i\leq k, \ \forall \ j\geq1,$ and\\
$\deg Q^{(j)}_i=d^{(j)}_i\rightarrow\infty$ as $j\rightarrow\infty$ for each $k+1\leq i\leq N,$\\
\end{center}
where $k$ is some integer such that $0\leq k\leq N-1$. \\
Since $\deg Q^{(j)}_i=d_i\ \forall \ 0\leq i\leq k, \ \forall j\geq1,$ we may assume that, for each $0\leq i\leq k,$ $\left\{Q^{(j)}_i\right\}_{j\geq1}$ converges uniformly on compact subsets of $\C^{N+1}$ to either a homogeneous polynomial $Q_i$ of degree $d_i$ with coefficients being bounded above by $1$ or to the zero polynomial.
Since $0\leq k\leq N-1$, we have
\begin{center}
$\bigcap^k_{i=0}\limits \left\{H_i:=Zero(Q_i)\right\}\neq\emptyset.$\\
\end{center}
Hence, there exists $\omega^{(0)}\in\bigcap^k_{i=0}\limits H_i$ with $||\omega^{(0)}||=1.$ We now consider two subcases.

\noindent
{\bf Subcase 1.1.}\  Assume that  $r=\max\left\{|\omega^{(0)}_0|,..,|\omega^{(0)}_N|\right\}<1.$ \\
$+)$ \,If $0\leq i\leq k,$ then
\begin{center}
$\lim_{j\rightarrow\infty}\limits Q^{(j)}_i(\omega^{(0)})=0.$\\
\end{center}
$+)$\, If $k+1\leq i \leq N,$ then, by remark $i),$ we have
\begin{center}
$|Q^{(j)}_i(\omega^{(0)})|\leq (d^{(j)}_i+1)^{N+1}r^{d^{(j)}_i}.$\\
\end{center}
Since $\lim_{j\rightarrow\infty}\limits d^{(j)}_i=\infty$ and $r<1$, it implies that
\begin{center}
$\lim_{j\rightarrow\infty}\limits Q^{(j)}_i(\omega^{(0)})=0$ for $k+1\leq i \leq N.$\\
\end{center}
Therefore, we get
\begin{center}
$\lim_{j\rightarrow\infty}\limits D(Q^{(j)}_0,..,Q^{(j)}_N)\leq \lim_{j\rightarrow\infty}\limits \sum^N_{i=0}\limits |Q^{(j)}_i(\omega^{(0)})|^2=0.$\\
\end{center}
This is a contradiction.

\noindent
{\bf Subcase 1.2.}\  Assume that $\max\left\{|\omega^{(0)}_0|,..,|\omega^{(0)}_N|\right\}=1.$\\
We may assume that $\omega^{(0)}=(1,0,\cdots,0)$. Set $\left\{\omega^{(j)}\right\}_{j\geq 1}$ such that
\begin{center}
$\omega^{(j)}_0=1-\frac{1}{\sqrt{d^{(j)}}}$ , $\omega^{(j)}_1=\cdots=\omega^{(j)}_N=\frac{1}{\sqrt N}\sqrt{\frac{2}{\sqrt {d^{(j)}}}-\frac{1}{d^{(j)}}},$\\
\end{center}
where $d^{(j)}=\min_{k+1\leq i \leq N}\limits d^{(j)}_i \,.$\\
$+)$ \,If $0\leq i\leq k,$ then
\begin{center}
$\lim_{j\rightarrow\infty}\limits Q^{(j)}_i(\omega^{(j)})=Q_i(\omega^{(0)})=0.$\\
\end{center}
$+)$ \ Suppose that $k+1\leq i \leq N.$ \\
Since $\lim_{j\rightarrow\infty}\limits d^{(j)}=\infty,$  there exists $j_0$ such that:
\begin{center}
$\max\left\{|\omega^{(j)}_0|,..,|\omega^{(j)}_N|\right\}=|\omega^{(j)}_0|=1-\frac{1}{\sqrt{d^{(j)}}}=r_j \mbox { for any } j>j_0.$  \\
\end{center}
By remark $i)$ and $iii)$, for each  $k+1\leq i \leq N,$ we have
\begin{align*}
|Q^{(j)}_i(\omega^{(j)})|&\leq (d^{(j)}_i+1)^{N+1}(1-\frac{1}{\sqrt{d^{(j)}}})^{d^{(j)}_i}\\
&\leq(d^{(j)}_i+1)^{N+1}(1-\frac{1}{\sqrt{d_i^{(j)}}})^{d^{(j)}_i}\rightarrow0 \mbox { as } j\rightarrow\infty.
\end{align*}
This is a contradiction by the same argument as above. 

\noindent
{\bf Case 2:} \,$q>N+1$.\\
By remark $ii)$ we have

$\delta< D(Q_1,..,Q_q)=\prod_{1\leq j_0<..<j_N\leq q}\limits D(Q_{j_0},..,Q_{j_N})$
$\leq CD(Q_{j_0},..,Q_{j_N})$\\
for any set $\left\{j_0,..,j_N\right\}\subset\left\{1,..,q\right\},$ where $C$ is a constant which depends only on $N$ and $q.$\\
By Case $1$, we have
\begin{center}
$\max\left\{\deg Q_{j_0},\cdots,\deg Q_{j_N}\right\}<M(\delta / C, N, N+1)$\\
\end{center}
for any set $\left\{j_0,..,j_N\right\}\subset\left\{1,..,q\right\}.$
So if we  define $$M(\delta, N, q):=M(\delta / C, N, N+1)$$ (this is well defined since
$C$ only depends on $N$ and $q$), then we have 
\begin{center}
$\max\left\{\deg Q_{1},\cdots,\deg Q_{q}\right\}<M(\delta, N, q)$.\\
\end{center}
\end{proof}

\begin{lem}\label{lemma 5n}
{\it Let natural numbers $N$ and $q \geqslant N+1$ be fixed.
Let $H_k^{(p)}\ (1 \leqslant k \leqslant q, \ p \geqslant 1)$ be moving hypersurfaces in $\P^N\big(\C\big)$ such that the following conditions are satisfied:

i) For each $1 \leqslant k \leqslant q, \ p \geqslant 1,$ the coefficients of the homogeneous polynomials $Q_k^{(p)}$ which define the $H_k^{(p)}$  are bounded above uniformly for all $p \geq 1$ on compact subsets of $D$,

ii) there exists $z_0 \in D$ such that 
$$inf_{p \in \N}\big\{D(Q_1^{(p)},...,Q_{q}^{(p)})(z_0) \big\} >  \delta > 0 \mbox \,.$$
Then, we have:

a) There exists a subsequence $\{j_p\}\subset \N$ such that for 
$1 \leqslant k \leqslant q$, $Q_k^{(j_p)}$ converge uniformly on compact subsets of $D$ to not identically zero homogenous polynomials
$Q_k$ (meaning that the $Q_k^{(j_p)}$ and $Q_k$ are homogenous polynomials in
$\widetilde{\mathcal H}_D[\omega_0, \cdots, \omega_N]$ of the same degree, and all their coefficients converge uniformly on compact subsets of $D$). 
Moreover we have that
$D\big(Q_1, \cdots, Q_{q}\big)(z_0) > \delta  > 0$,
 the hypersurfaces $Q_1(z_0), \cdots, Q_{q}(z_0)$ are located in general position and the moving hypersurfaces
 $Q_1(z), \cdots, Q_{q}(z)$ are located in (weakly) general position.

b) There exists a subsequence $\{j_p\}\subset \N$ and $r = r(\delta) > 0$  such that 
\begin{equation*}
\inf\{D\big(Q_1^{(j_p)}, \cdots, Q_{q}^{(j_p)}\big)(z) \big| p \geqslant 1\} > \dfrac{\delta}{4}, \ \forall z \in B(z_0, r).
\end{equation*}}
\end{lem}

\begin{proof} 
Let  $d_k^{(p)}=\deg Q_k^{(p)}$ be the degree of the non identically vanishing homogenous polynomial $Q_k^{(p)} \,(1 \leqslant k \leqslant q, p \geqslant 1)$. Then we have
 \begin{equation*}
Q_k^{(p)}(z)(\omega) = \sum_{| I | = d_k^{(p)}} a_{kpI}(z).\omega^I, 
\end{equation*}
where $I = (i_1, .., i_{N+1}), |I|= i_1+\cdots+i_{N+1}$ and $a_{kpI}(z)$ are holomorphic functions which are bounded above uniformly  for all $p \geq 1$ on compact subsets of $D$. 
Since the coefficients of the polynomials $Q_k^{(j_p)}$ are bounded above uniformly for all $p \geq 1$ on compact subsets of $D$, there exists $c >0$ such that $|a_{kpI}(z_0)| \leq c$ for all $k,p,I$. Define 
homogenous polynomials 
 \begin{equation*}
 \tilde Q_k^{(p)}(z)(\omega)   :=         \frac{1}{c}  Q_k^{(p)}(z) (\omega)
  \end{equation*}
  Then the $ \tilde Q_k^{(p)}(z)(\omega)$ satisfy the condition 
  \begin{equation}\label{n1}
  inf_{p \in \N}\big\{D(\tilde Q_1^{(p)},...,\tilde Q_{q}^{(p)})(z_0) \big\} > \tilde \delta > 0 \mbox \,,
  \end{equation}
  with $\tilde \delta :=(\frac{1}{c})^{2 
  \left( 
  \begin{array}{c}
  q \\
  N+1
  \end{array}
   \right)
  } \delta$.
  By Lemma \ref{lemma 4}, we have 
  $$\max\{\deg \tilde Q_1^{(p)}(z_0), \cdots, \deg \tilde Q_q^{(p)}(z_0)\} < M(\tilde \delta).$$
  Since by equation (\ref{n1}) none of the homogenous polynomials 
  $Q_k^{(p)}(z_0) \,(1 \leqslant k \leqslant q, p \geqslant 1)$ can be
 the zero polynomial, we get that 
 $$\max\{\deg \tilde Q_1^{(p)}(z), \cdots, \deg \tilde Q_q^{(p)}(z)\} < M(\tilde \delta)$$
 for all $z \in D$.
 So if again 
  \begin{equation*}
\tilde Q_k^{(p)}(z)(\omega) = \sum_{| I | = d_k^{(p)}} \tilde a_{kpI}(z).\omega^I, 
\end{equation*}
after passing to a subsequence $\{j_p\}\subset \N$ (which we denote
for simplicity again by $\{p\}\subset \N$), we can assume that 
$d_k^{(p)}= d_k$ for $1 \leq k \leq q$. So if we still multiply by $c$, we get
  \begin{equation*}
Q_k^{(p)}(z)(\omega) = \sum_{| I | = d_k}  a_{kpI}(z).\omega^I. 
\end{equation*}
Now, since the $a_{kpI}(z) $ are locally bounded uniformly  for all $p \geq 1$  on $D$,
by using Montel's theorem and a standard diagonal argument with respect to an exaustion of $D$ with compact subsets, 
after passing to a subsequence $\{j_p\}\subset \N$ (which we denote
for simplicity again by $\{p\}\subset \N$),
we also can assume that $\{ a_{kpI}(z)\}_{p=1}^\infty$ converges uniformly on compact subsets of $D$ to $ a_{kI}$ for each $k, I.$  Denote by
\begin{equation*}
 Q_k(z)(\omega) = \sum_{| I | =d_k}  a_{kI}(z).\omega^I.
\end{equation*}
Then 
\begin{equation}\label{n2}
D\big(Q_1, \cdots, Q_{q}\big)(z_0)  \geqslant \liminf_{p \longrightarrow \infty}\limits D\big(Q_1^{(p)}, \cdots, Q_{q}^{(p)}\big)(z_0)> \delta  > 0,
\end{equation}
hence,  the hypersurfaces $Q_1(z_0), \cdots, Q_{q}(z_0)$ are located in general position and so the moving hypersurfaces
 $Q_1(z), \cdots, Q_{q}(z)$ are located in (weakly) general position
 (and in particular all the $Q_1(z),...,Q_q(z)$ are not identically zero),
 which proves a).
 
 Moreover, by equation (\ref{n2}), 
  there exists $r = r(\delta)$ such that 
$$D\big(Q_1, \cdots, Q_{q}\big)(z) > \dfrac{\delta}{2}, \ \forall z \in B(z_0, r).$$
Since $\{Q_k^{(p)}\}$ converges uniformly on compact subsets of $D$ to $Q_k,$  after shrinking $r$ a bit if necessary, there exists $M$ such that 
$$D\big(Q_1^{(p)}, \cdots, Q_{q}^{(p)}\big)(z) > \dfrac{\delta}{4}, \ \forall z \in B(z_0, r), \ p > M,$$
 which proves b).
\end{proof}

\begin{lem}\label{lemma 6n} 
{\it Let $\{f^{(p)}\}$
be a sequence of meromorphic mappings of 
 a domain $D$ in
$\C^n$ into $\P^N(\C)$ and let $S$ be a  closed subset of $D$ with $\Lambda^{2n-1}(S)=0$.
Suppose that $\{f^{(p)}\}$ meromorphically converges on $D-S$ to a meromorphic mapping 
$f$ of $D-S$ into $\P^N(\C)$.
Suppose that, for each $f^{(p)}$, there exist $N+1$
moving hypersurfaces $H_1(f^{(p)}),\cdots,H_{N+1}(f^{(p)})$ in 
$\P^N(\C),$ where  the moving hypersurfaces $H_i(f^{(p)})$ may depend on $f^{(p)},$
such that the following three conditions are satisfied:

i) For each $1 \leqslant k \leqslant N+1,$ the coefficients of homogeneous polynomial $Q_k(f^{(p)})$ which define $H_k(f^{(p)})$ for all $f^{(p)}$  are bounded above uniformly  for all $p \geq 1$ on compact subsets of $D$.

ii) \ There exists  $z_0 \in D$ such that \\
$$\inf\{D\big(Q_1(f^{(p)}), \cdots, Q_{N+1}(f^{(p)})\big)(z_0) \big| p \geqslant 1\} > 0.$$

iii)\ The  $2(n-1)$-dimensional Lebesgue areas of $\big(f^{(p)}\big)^{-1}\big(H_k(f^{(p)})\big) \cap K$ $(1 \leqslant k \leqslant N+1, \ p \geqslant 1)$ counting multiplicities are bounded above
 for any fixed compact subset $K$ of $D$.
 
Then we have:

a) $\{f^{(p)}\}$ has a meromorphically convergent subsequence on $D.$

b) If, moreover, $\{f^{(p)}\}$
is a sequence of holomorphic mappings 
of a domain $D$ in $\C^n$ into $\P^N(\C)$
and condition iii) is sharpened to
$$f^{(p)}(D) \cap H_k(f^{(p)}) = \emptyset \ (1\leq  k\leq N+1, \ p\geq 1),$$
then $\{f^{(p)}\}$ has a subsequence which converges uniformly on compact subsets of $D$ to a holomorphic mapping of $D$ to $\P^N\big(\C\big).$ }

\end{lem}

\begin{proof}  
By Lemma \ref{lemma 5n} and conditions i) and ii),
after passing to a  subsequence, we may assume that for 
$1 \leqslant k \leqslant N+1$, $Q_k(f^{(p)})$ converge uniformly on compact subsets of $D$ to
$Q_k$, in particular they have common degree $d_k$. Moreover, $Q_1,..., Q_{N+1}$ are located in
(weakly) general position. Denote by $H_1,...,H_{N+1}$ the corresponding
moving hypersurfaces.

By Lemma \ref{lemma 2} and condition iii), after passing to a subsequence,
we may assume that for every $1 \leq k \leq N+1$, the divisors 
$$\{\nu(f^{(p)},H_k(f^{(p)}))\} = \big(f^{(p)}\big)^{-1}\big(H_k(f^{(p)})\big)\:(p \geqslant 1)$$ are convergent (in the sense of convergence of divisors in $D$). 

By a standard diagonal argument we may assume that 
$D=B(R)$, and that
$\{f^{(p)}\}$ meromorphically converges on $B(R)-S$ to a meromorphic mapping
$f : B(R)-S \to \P^N(\C)$.

We prove that there exists $k_0 \in \{1,...,N+1\}$ such that 
$f(D - S) \not\subset H_{k_0}$, more precisely that for any
representation $f=(f_0:...:f_N)$ of $f : D - S \to \P^N(\C)$ (admissible or not) we have $Q_{k_0}(f_0,...,f_N) \not\equiv 0$:\\
$E= \{ z \in D : f_0(z) = f_1(z) =...= f_N(z) = 0\}$ is a proper analytic subset. Since $Q_1,..., Q_{N+1}$ are located in
(weakly) general position,
 there exists  $z \in D$ such that  the system of equations 
\begin{equation*}
\left \{ \begin{array}{ll}
Q_{k} (z) \big(\omega_0, \cdots, \omega_N\big) = 0\\
1 \leqslant k \leqslant N+1
\end{array} \right.
\end{equation*}
has only the trivial solution $\omega = \big(0, \cdots, 0\big)$ in $\C^{N+1}$.
 But since then the same is true for the generic point $z \in D$ it is true in
 particular for the generic point $z \in D - E$. So for such point $z$ there
 exists some $k \in \{1,...,N+1\}$ such that  $Q_{k}(z)(f_0(z),...,f_N(z)) \not= 0$. In order to simplify notations, from now on we put:
 $$ Q^{(p)}:= Q_{k_0}(f^{(p)}), \: Q:=Q_{k_0},\: H^{(p)}:= H_{k_0}(f^{(p)}),\:
 H:=H_{k_0}, \:d:=d_{k_0}.$$
 
 Let $z_1$ be any point of $S$.
By \cite{S1} Theorem  3.6, for any $r$ $(0<r<\tilde R = R-||z_1||)$, we can choose holomorphic functions 
$h^{(p)}$ $\not\equiv 0$ and 
$h$ $\not\equiv 0$ on $B(z_1,r)$ such that $\nu(f^{(p)},H^{(p)})=\nu_{h^{(p)}},\nu=\nu_h$
for the limit $\nu$ of 
$\{\nu(f^{(p)},H^{(p)})\}$ and
$\{h^{(p)}\}$ converges uniformly on compact subsets of $B(z_1,r)$ to $h$.
Moreover, each $f^{(p)}$ has an admissible representation on $B(z_1,r)$
$$f^{(p)}=(f_0^{(p)}:f_1^{(p)}:...:f_N^{(p)})$$
with suitable holomorphic functions $f_i^{(p)}$ $(0\le i\le N)$ on $B(z_1,r)$.

Let $z$ be a point in $B(z_1,r)-(S \cup \{h=0\})$.
Choose a simply connected relatively compact neighborhood $W_z$ 
of $z$ in $B(z_1,r)-(S \cup \{h=0\})$ such that
there exists a sequence $\{u_z^{(p)}\}$ of nonvanishing holomorphic functions on $W_z$
such that $\{u_z^{(p)}f_i^{(p)}\}\to f_i^z \ (0\le i \le N) \text { on } W_z $
and $f=(f_0^z:f_1^z:...:f_N^z)$ on $W_z.$ It may be assumed that 
$h^{(p)}$ $(p \ge 1)$ has no zero on $W_z$.
We have $Q^{(p)}(f_0^{(p)},...,f_N^{(p)})=v^{(p)}h^{(p)},$
where $v^{(p)}$ is a nonvanishing holomorphic function on $B(z_1,r).$ 
This implies that $Q^{(p)}(u_z^{(p)}f_0^{(p)},...,u_z^{(p)}f_N^{(p)})\ne 0$
on $W_z$,
since $Q^{(p)}$ is a homogeneous polynomial, and we have
$$Q^{(p)}(u_z^{(p)}f_0^{(p)},...,u_z^{(p)}f_N^{(p)})\to Q(f_0^z,...,f_N^z)$$
on $W_z$ since $Q^{(p)}$ converge uniformly on compact subsets of $D$ to $Q$.
Since $f(D-S)\not\subset H,$ it implies that $Q(f_0^z,...,f_N^z)\not\equiv 0$ on $W_z$, and
hence $Q(f_0^z,...,f_N^z)\ne 0$ on $W_z$.

We recall that the $Q^{(p)}$, $p \geq 1$, and $Q$ have common degree $d$. Since 
$$Q^{(p)}(u_z^{(p)}f_0^{(p)},...,u_z^{(p)}f_N^{(p)}) \text { tends to } Q(f_0^z,...,f_N^z)\text { on } W_z \text { and }$$ $$Q^{(p)}(u_z^{(p)}f_0^{(p)},...,u_z^{(p)}f_N^{(p)}) =(u_z^{(p)})^d\cdot v^{(p)}\cdot h^{(p)},$$
it follows that $(u_z^{(p)})^d\cdot v^{(p)}\cdot h^{(p)}$ tends to $ Q(f_0^z,...,f_N^z)\text { on } W_z$.
Since $v^{(p)}\ne 0$ on $B(z_1,r)$,  $v^{(p)}=(k^{(p)})^d$,
where $k^{(p)}$ is a nonvanishing holomorphic function  on $B(z_1,r)$.
We have
$$(u_z^{(p)})^d\cdot (k^{(p)})^d=(u_z^{(p)}\cdot k^{(p)})^d
\to \dfrac{Q(f_0^z,...,f_N^z)}{h}\text{ on }W_z.$$

Define $F^z$ such that
$$(F^z)^d:=\dfrac{Q(f_0^z,...,f_N^z)}{h}\text { on } W_z.$$

We can do this because $\dfrac{Q(f_0^z,...,f_N^z)}{h}\ne 0$ on $W_z$.
So $(u_z^{(p)}\cdot k^{(p)})^d\to (F^z)^d$
on $W_z$, hence $(\dfrac{u_z^{(p)}\cdot k^{(p)}}{F^z})^d$
tends to $1$ on $W_z$.
Therefore, there exist infinite (or empty) subsets $\{N_j^z\}_{j=0}^{d-1}$ of $\N$ such that
$$\N \ \text { is a disjoint union of sets } N_j^z \ \text { and }$$
$$\{\dfrac {u_z^{(p)}\cdot k^{(p)}}{F^z} \}_{p\in N_j^z} \to \theta _j = e^{i \cdot {\frac {2\pi j}{d}}}\text { for each }
0\le j\le d-1.$$

This implies that $\{\dfrac {f_i^{(p)}}{k^{(p)}}\}_{p\in N_j^z} \to \dfrac {F_i^z}{\theta_j}$ on $W_z$, where
$F_i^z= \dfrac {f_i^z}{F^z}$ on $W_z.$

Take $a \in B(z_1,r)-(S \cup \{h=0\})$. Then $\{\dfrac {f_i^{(p)}}{k^{(p)}}\}_{p\in N_j^a} \to \dfrac {F_i^a}{\theta_j}$ on $W_a$
for each $0 \le j \le d-1$. 

Take $b \in B(z_1,r)-(S \cup \{h=0\})$ such that $W_a \cap W_b \ne \emptyset.$
We will prove that $\{\dfrac {f_i^{(p)}}{k^{(p)}}\}_{p\in N_j^a} \to {\dfrac {F_i^b}{\theta_j}}\cdot c$
for each $ 0\le j\le d-1.$ 
Indeed, without loss of generality we may assume that $f_0^a \not\equiv 0$ on $W_a$.
Then $f_0^x \not\equiv 0$ on $W_x$ for each $x \in B(z_1,r)-(S \cup \{h=0\}).$ Hence $F_0^x \not\equiv 0$ on $W_x$ for each 
$x \in B(z_1,r)-(S \cup \{h=0\}).$

Consider $|N_j^a|=\infty, $ where $|.|$ denotes the cardinality of a set.

Assume that there exist $N_1^b, N_2^b$ such that 
for $\tilde N := N_j^a \cap N_1^b$ and  $\tilde{\tilde N} := N_j^a \cap N_2^b$
we have 
$|\tilde N | = |\tilde{\tilde N} |=
\infty$.
Since $\{\dfrac {f_0^{(p)}}{k^{(p)}}\}_{p\in \tilde N \subset N_1^b} \to \dfrac {F_0^b}{\theta_1}$ on $W_b$
and $\{\dfrac {f_0^{(p)}}{k^{(p)}}\}_{p\in \tilde N \subset N_j^a} \to \dfrac {F_0^a}{\theta_j}$ on $W_a$, we have
$\dfrac {F_0^b}{\theta_1}= \dfrac {F_0^a}{\theta_j}$ on $W_a \cap W_b.$ Similarly, $\dfrac {F_0^b}{\theta_2}= \dfrac {F_0^a}{\theta_j}$ on $W_a \cap W_b.$ This is a contradiction. 
Thus every infinite subset $N_j^a$ intersects and  
only intersects infinitely with the subset $N_{\alpha(j)}^b$. Moreover, $|N_j^a\Delta N_{\alpha(j)}^b| < \infty,$ where $A\Delta B = (A - B)\cup (B - A)$ for arbitrary sets $A,B.$

From this it follows that there exists a bijection $\alpha : \{0,1,...,d-1\}\to \{0,1,...,d-1\}$ such that
$$N_j^a = \emptyset \text { if and only if } N_{\alpha(j)}^b = \emptyset,$$
$$\text { if } |N_j^a| =\infty \text { then } |N_j^a\Delta N_{\alpha(j)}^b| < \infty.$$

On the other hand, since $\{\dfrac {f_0^{(p)}}{k^{(p)}}\}_{p\in N_j^a \cap N_{\alpha(j)}^b} \to \dfrac {F_0^a}{\theta_j}$ 
on $W_a$
and\\  
$\{\dfrac {f_0^{(p)}}{k^{(p)}}\}_{p\in N_j^a \cap N_{\alpha(j)}^b}\to \dfrac {F_0^b}{\theta_{\alpha(j)}}$ on $W_b$, 
we have
$\dfrac {F_0^a}{\theta_j}= \dfrac {F_0^b}{\theta_{\alpha(j)}}$ on $W_a \cap W_b.$ 
This means that $F_0^a=F_0^b \cdot \dfrac {\theta_j}{\theta_{\alpha(j)}}$ on $W_a \cap W_b$ for each $0\le j \le d-1,$
and hence,  $c_b:=\dfrac {\theta_j}{\theta _ {\alpha(j)}} $ is a constant independant of $j$, $0\le j \le d-1.$
It implies that $\{\dfrac {f_i^{(p)}}{k^{(p)}}\}_{p\in N_j^a \cap N_{\alpha(j)}^b} \to \dfrac {F_i^b}{\theta_{\alpha(j)}}=
\dfrac {F_i^b}{\theta_j}\cdot c_b$ on $W_b,$ and hence,
$$\{\dfrac {f_i^{(p)}}{k^{(p)}}\}_{p\in N_j^a } \to \dfrac {F_i^b}{\theta_j}\cdot c_b \mbox{ on } W_b.$$

Applying this procedure a finite number of times, we have $$\{\dfrac {f_i^{(p)}}{k^{(p)}}\}_{p\in N_j^a } \to \dfrac {F_i^x}{\theta_j}\cdot c_x$$ on 
$W_x$ for each $x \in B(z_1,r)-(S \cup \{h=0\})$ and for each $0\le j \le d-1\,.$
Indeed, by the assumption on the Hausdorff dimension of $S$
and by Corollary \ref{HDC}, the open set  
 $B(z_1,r)-(S \cup \{h=0\})$ is pathwise connected, and such a path between $a$ and $x$, which is compact as the image of a closed interval under a continuous map,  can be covered by a finite number of such neighborhoods $W_y$ with 
 $y \in B(z_1,r)-(S \cup \{h=0\})$. And since the limit is unique if it exists, it does not
 depend on the choice of the path.
For $p\in N_j^a$ put $\tilde f_i^{(p)} = f_i^{(p)}\cdot  \dfrac {\theta_j}{k^{(p)}}\ (0\le i \le N).$
Then $f^{(p)}= (\tilde f_0^{(p)},...,\tilde f_N^{(p)})$ for all $p \in N_j^a$ and $0\le j \le d-1$ and 
$\{\tilde f_i^{(p)}\}_{p=1}^\infty\to F_i^x\cdot c_x$ on $ W_x$ for each $0\le i \le N.$
Note that if $W_x \cap W_y \ne\emptyset \ (x,y\in B(z_1,r)-(S \cup \{h=0\}))$ then $ F_i^x\cdot c_x=F_i^y\cdot c_y$ for each $0\le i \le N.$

Define the function $F_i : B(z_1,r)-(S \cup \{h=0\}) \to \C$ given by $F_i|_{W_z}=F_i^z\cdot c_z.$ Then $\{\tilde f_i^{(p)}\}_{p=1}^\infty \to F_i$
on $B(z_1,r)-(S \cup \{h=0\})$ for each $0\le i \le N.$

We now prove that the sequence $\{f^{(p)}\}_{p=1}^\infty$ meromorphically converges on  $B(z_1,r)$  to  some meromorphic mapping $\tilde F = (\tilde F_0,...,\tilde F_N). $ 
Indeed, let $z^{(0)}$ be any point of $ S_1 = S \cup \{h=0\}.$  Since 
 $\Lambda^{2n-1}(S_1)=0$, by Lemma~\ref{HD} a) there exists a complex line $l_{z^{(0)}}$ passing through $z^{(0)}$ such that $\Lambda^1(S_1 \cap l_{z^{(0)}})=0$. Put $ l_{z^{(0)}}= \{z^{(0)}+z\cdot u: z \in \C\}$. Then by Lemma~\ref{HD} b) there exists $R>0$ such that 
$$\mathcal C_0 = \{z^{(0)} +R\cdot e^{i\theta}\cdot u : \theta \in [0, 2\pi]\ \}$$
satisfying $\mathcal C_0 \subset B(z_1,r)$ and $\mathcal C_0 \cap S_1 = \emptyset$.
By the maximum principle, it implies that the sequence $\{\tilde f_i^{(p)}(z^{(0)})\}$ converges. Put
$\lim_{p\to \infty}\tilde f_i^{(p)}(z^{(0)})=\tilde  F_i(z^{(0)}).$ This means that the mapping $F_i$ extends over $B(z_1,r)$ to the mapping $\tilde F_i.$

We now prove that the sequence $\{\tilde f_i^{(p)}(z)\}_{p=1}^\infty $ converges uniformly on compact subsets of $B(z_1,r)$ to $\tilde F_i(z).$
Indeed, assume that $\{z^{(j)}\}\subset B(z_1,r)$ converges to $z^{(0)} \in B(z_1,r)$. As above, there exists a circle
$\mathcal C_0 = \{z^{(0)} +R\cdot e^{i\theta}\cdot u : \theta \in [0,2\pi] \ \}\subset B(z_1,r)$
such that $\mathcal C_0 \cap S_1 = \emptyset$. Since $\mathcal C_0$ is a compact subset of $B(z_1,r)-S_1,$
there exists $\epsilon_0 >0$ such that 
$$V(\mathcal C_0,\epsilon_0)=\{z\in \C^n : dist(z,\mathcal C_0)<\epsilon_0 \} \Subset B(z_1,r)-S_1.$$
Consider the circles $\mathcal C_j = \{z^{(j)} +R\cdot e^{i\theta}\cdot u : \theta \in [0,2\pi] \ \}.$
It is easy to see that $dist(\mathcal C_0,\mathcal C_j) = ||z^{(j)}-z^{(0)}|| \to 0$ as $j\to \infty.$
Thus, without loss of generality, we may assume that $\mathcal C_j \subset  V(\mathcal C_0,\epsilon_0) \Subset B(z_1,r)-S_1.$
By the hypothesis, $\forall \epsilon >0, \exists N= N(\epsilon)$ such that 
$$\sup \{ ||\tilde f^{(p)}_i(z)-F_i(z)|| : z \in V(\mathcal C_0,\epsilon_0), \ p\ge N\} <\epsilon .$$
By the maximum principle, we have
$\limsup_{j \to \infty}||\tilde f^{(j)}_i(z^{(j)})-F_i(z^{(j)})|| =0.$
This implies that the sequence $\{\tilde f_i^{(p)}\}_{p=1}^\infty $ converges uniformly on compact subsets of $B(z_1,r)$ to $\tilde F_i.$ This finishes the proof of part a) of the lemma.\\

In order to prove part b), we first remark that it suffices to prove that
 $\{f^{(p)}\}$
has a subsequence which converges locally uniformly on  $D$ to a holomorphic mapping $f$ of $D$ to $\P^N\big(\C\big),$ that means
that after passing to a subsequence we have:\\
 Let $z_1$ be any point of $D$. Then there exists $r >0$ and,  
for each $f^{(p)}$  a holomorphic representation on $B(z_1,r)$
$$f^{(p)}=(f_0^{(p)}:f_1^{(p)}:...:f_N^{(p)})$$
with suitable holomorphic functions $f_i^{(p)}$ $(0\le i\le N)$ without common zeros on $B(z_1,r)$, such that $\{f_i^{(p)}\}\to f_i \ (0\le i \le N) $ uniformly on $B(z_1,r)$
and $f=(f_0:f_1:...:f_N)$ is a holomorphic map on $B(z_1,r)$, that means
 the $f_i$ $(0\le i\le N)$ are without common zeros on $B(z_1,r)$.

By part a) we know that  $\{f^{(p)}\}$
has a subsequence which converges meromorphically on  $D$ to a meromorphic mapping $f$ of $D$ to $\P^N\big(\C\big),$ that means
that after passing to a subsequence we have:\\
 Let $z_1$ be any point of $D$. Then there exists $r >0$ and,  
for each $f^{(p)}$  an admissible representation on $B(z_1,r)$
$$f^{(p)}=(f_0^{(p)}:f_1^{(p)}:...:f_N^{(p)})$$
with suitable holomorphic functions $f_i^{(p)}$ $(0\le i\le N)$  on $B(z_1,r)$, such that $\{f_i^{(p)}\}\to f_i \ (0\le i \le N) $ uniformly on $B(z_1,r)$
and $f=(f_0:f_1:...:f_N)$ is a meromorphic map on $B(z_1,r)$.
Observing that the admissible representations of the holomorphic maps
$f^{(p)}=(f_0^{(p)}:f_1^{(p)}:...:f_N^{(p)})$ are automatically without common zeros, the only thing which remains to be proved is that 
under the conditions of part b) we have 
$$E= \{z \in B(z_1,r) : f_0(z)=f_1(z)=...=f_N(z)=0\} = \emptyset \,.$$
We also recall that by the proof of part a) we have that: There exists $k_0 \in \{1,...,N+1\}$ such that 
$Q^{(p)}=Q_{k_0}(f^{(p)})$, $p \geq 1$ converge uniformly on compact subsets of $D$ to $Q=Q_{k_0}$, and 
$f(D - S) \not\subset H_{k_0}$, more precisely that for any
representation $f=(f_0:...:f_N)$ of the meromorphic map $f : D  \to \P^N(\C)$ (admissible or not) we have 
\begin{equation} \label{null}
Q(f_0,...,f_N) \not\equiv 0\,.
\end{equation}

Now we can end the proof with an easy application of Hurwitz's theorem:
By the condition of b) we have that for all $p \geq 1$, 
$$Q^{(p)}(f_0^{(p)},...,f_N^{(p)}) \not= 0$$ on $B(z_1,r)$.
 And we also have
that $$Q^{(p)}(f_0^{(p)},...,f_N^{(p)})\to Q(f_0,...,f_N)$$ uniformly on compact subsets of $B(z_1,r)$. By equation (\ref{null}) and the Hurwitz's theorem
we get that $Q(f_0,...,f_N) \not= 0$ on $B(z_1,r)$. But since $Q$ is a homogenous polynomial this implies that 
$$E= \{z \in B(z_1,r) : f_0(z)=f_1(z)=...=f_N(z)=0\} = \emptyset \,.$$
\end{proof}

We remark that the following corollary of part a) of the previous lemma generalizes the Proposition $3.5$ in $[2]$.
\begin{cor}
{\it Let $\{f^{(p)}\}$
be a sequence of meromorphic mappings of 
 a domain $D$ in
$\C^n$ into $\P^N(\C)$ and let $S$ be a closed subset of $D$ with $\Lambda^{2n-1}(S)=0$.
Suppose that $\{f^{(p)}\}$ meromorphically converges on $D-S$ to a meromorphic mapping 
$f$ of $D-S$ into $\P^N(\C)$.
If there exists a 
moving hypersurface $H$ in $\P^N(\C)$
such that $f(D-S)\not\subset H$ and 
$\{\nu(f^{(p)},H)\}$
is a convergent sequence of divisors on $D$, then 
$\{f^{(p)}\}$ is meromorphically convergent on $D$.}
\end{cor}

\begin{lem}\label{lemma 9}(\cite[Theorem 2.5]{TTH})
{\it Let $\mathcal F$ be a family of holomorphic mappings of a domain $D$ in $\C^n$ onto $\P^N\big(\C\big)$. Then the family $\mathcal F$ is not normal on $D$ if and only if there exist a compact subset $K_0 \subset D$ and sequences $\{f_i\} \subset \mathcal F, \{z_i\} \subset K_0, \{r_i\}\subset \R$ with $r_i> 0$ and $r_i \longrightarrow 0^+$, and $\{u_i\} \subset \C^n$ which are unit vectors such that
\begin{equation*}
g_i(\xi) := f_i\big(z_i +r_iu_i\xi),
\end{equation*}
where $\xi \in \C$ such that $z_i + r_iu_i\xi \in D$, converges uniformly on compact subsets of $\C$ to a nonconstant holomorphic map $g$ of $\C$ to $\P^N\big(\C\big)$.}
\end{lem}

\begin{lem}\label{lemma 10}(See \cite[Theorem 4']{Noc})
{\it Suppose that $q \geqslant 2N + 1$ hyperplanes $H_1, \cdots, H_q$ are given in general position in $\P^N\big(\C\big)$ and $q$ positive intergers (may be $\infty$) $m_1, \cdots, m_q$ are given such that
\begin{equation*}
\sum_{i = 1}^q \bigg(1 - \dfrac{N}{m_j}\bigg) > N+1. 
\end{equation*}
Then there does not exist a nonconstant holomorphic mapping
$f : \C \longrightarrow \P^N\big(\C\big)$ 
such that $f$ intersects $H_j$ with multiplicity at least $m_j\ (1\leq j \leq q).$}
\end{lem}

\begin{lem}\label{lemma 11}
{\it Let $P_0, \cdots, P_N$ be $N + 1$ homogeneous polynomials of common degree in $\C[x_0, \cdots, x_n]$. Let  $\{Q_j\}_{j=1}^q\ (q \geqslant N+1)$ be homogeneous polynomials in  
$\mathcal S\big(\{P_i\}_{i=0}^N\big)$ such that\\
$$D\big(Q_1, \cdots, Q_q\big) =
 \prod_{1 \leq j_0 < \cdots < j_N \leq q} \inf_{||\omega|| = 1} \bigg(\big|Q_{j_0}(\omega)\big|^2 + \cdots + \big|Q_{j_N}(\omega)\big|^2\bigg) > 0,$$
where $Q_j(\omega)=\sum_{| I | = d_j} a_{jI}.\omega^I$.

Then $\{Q_j\}_{j=1}^q$ are located in general position in $\mathcal S\big(\{P_i\}_{i=0}^N\big)$ and  $\{P_i\}_{i=0}^N$ are located in general position in $P^N(\C)$. (cf. Sec 2.3 and 2.6)}.
\end{lem}

\begin{proof} a) Suppose that $\{Q_j\}_{j=1}^q$  are not  located in general position in $S\big(\{P_i\}_{i=0}^N\big).$ Then there exist $N+1$ polynomials in $\{Q_j\}_{j=1}^q$ which are not linearly independent. Without loss of generality we may assume that
\begin{equation*}
Q_{N + 1} = \sum_{j = 1}^N b_jQ_j \; (b_j \in \C).
\end{equation*}
Then
\begin{equation*}
\begin{aligned}
X &= \{\omega \in \C^{N+1} \big| Q_1(\omega) = \cdots = Q_N(\omega) = Q_{N+1}(\omega) = 0\}\\
    &= \{ \omega \in \C^{N+1} \big| Q_1(\omega) = \cdots = Q_N(\omega) = 0\}
\end{aligned}
\end{equation*}
is an analytic subset in $\C^{N+1}$. Since $\dim X \geqslant 1$, there exists $\omega_0 \not= 0$ in $\C^{N+1}$ such that
\begin{equation*}
Q_1(\omega_0) = \cdots = Q_N(\omega_0) = Q_{N+1}(\omega_0) = 0.
\end{equation*}
Moreover, since  $\{Q_j\}_{j=1}^q$ are all homogenous polynomials, we may assume that $||\omega_0|| = 1$.
Thus, we have $$| Q_1(\omega_0)|^2 + \cdots + | Q_{N+1}(\omega_0)|^2 = 0\,,$$
and, hence,
$$D\big(Q_1, \cdots Q_q\big) = 0.$$  This is a contradiction.

b) Suppose that $\{P_i\}_{i=0}^N$ are not located in general position in $P^N(\C)$. Then there exists $\omega_0 \not= 0$ in $\C^{N+1}$ such that
\begin{equation*}
P_0(\omega_0) = \cdots = P_N(\omega_0) = 0.
\end{equation*}
Therefore, we have $Q_j(\omega_0)=0$ for any $1\leq j \leq q,$ and thus 
again 
$$D\big(Q_1, \cdots Q_q\big) = 0.$$  This is a contradiction.
\end{proof}

\begin{lem}\label{lemma 12}
{\it Let $f = (f_0 : \cdots : f_N) : \C \longrightarrow \P^N\big(\C\big)$ be a holomorphic mapping and $\{P_i\}_{i=0}^N$ be $N+1$ homogeneous polynomials in general position of common degree in $\C[\omega_0, \cdots, \omega_N]$. Assume that  $$F = (F_0 : \cdots : F_n) : \P^N\big(\C\big) \longrightarrow \P^N\big(\C\big)$$ is the mapping defined by
\begin{equation*}
F_i(\omega) = P_i\big(\omega\big),(0\leq i \leq N).
\end{equation*}
Then, $F(f)$ is a constant map if only if $f$ is a constant map.}
\end{lem}

\begin{proof} Since $\{P_i\}_{i=0}^N$ are $N+1$ homogeneous polynomials in general position of common degree in $\C[\omega_0, \cdots, \omega_N]$, $F$ is a morphism. Suppose that $F(f) = a,$ where $a = (a_0 : \cdots : a_n)\in \P^N\big(\C\big).$  We have $f(\C) \subset F^{-1}(a). $
Suppose that $dimF^{-1}(a) \geqslant 1$. Take $H$ any hyperplane in $\P^N\big(\C\big)$ with $a \not\in H$. Then  $F^{-1}(H)$  is a hypersurface in $\P^N\big(\C\big)$ since  the $\{P_i\}_{i=0}^N$
are in general position, so in particular they are not linearly dependant.
By Bezout's theorem there exists a point  $\omega_0 \in F^{-1}(a) \cap F^{-1}(H)$. Hence, 
$a= F(\omega_0) \in H$. This is a contradiction.
Therefore, $dimF^{-1}(a) =0$, so $F^{-1}(a)$ is a finite set. Since $f$ is continuous and $f(\C) \subset F^{-1}(a)$, it must be a constant map.
\end{proof}

\begin{lem}\label{lemma 13}
{\it Let $P_0, \cdots, P_N$ be $N + 1$ homogeneous polynomials of common degree in $\C[\omega_0, \cdots, \omega_N]$ and  $\{Q_j\}_{j=1}^q\ (q \geqslant 2N+1)$ be homogeneous polynomials in 
$\mathcal S\big(\{P_i\}_{i=0}^N\big)$ such that
\begin{equation*}
D\big(Q_1, \cdots, Q_q\big) > 0.
\end{equation*}
Assume that $m_1, \cdots, m_q$ are positive intergers (may be $\infty$)  such that
\begin{equation*}
\sum_{j=1}^q \bigg(1 - \dfrac{N}{m_j}\bigg) > N+1.
\end{equation*}
Then there does not exist a nonconstant holomorphic mapping
$$f : \C \longrightarrow \P^N\big(\C\big)$$ 
such that $f$ intersects $Q_j$ with multiplicity at least $m_j \ (1\leq j \leq q)$.}
\end{lem}

\begin{proof} Suppose that $f : \C \longrightarrow \P^N\big(\C\big)$ is a holomorphic mapping such that $f$ intersects $Q_i$ with multiplicity at least $m_i \; (1\leq i\leq q).$  For each $1\leq i\leq q,$ we define
\begin{equation*}
Q_j = \sum_{i=0}^N b_{ji}P_i  
\end{equation*}
and
\begin{equation*}
H_j = \bigg\{\omega \in \C^{N+1} \bigg| \sum_{i=0}^N b_{ji}\omega_i = 0\bigg \}.
\end{equation*}
Let $\widetilde{f} = \big(f_0, \cdots, f_N\big)$ be an admissible representation of $f$ on $\C$ (i.e. the $f_0,...,f_N$ have no common zeros), and denote $F = (P_0\big(\widetilde{f}\big): ...: P_N\big(\widetilde{f}\big)).$  By Lemma \ref{lemma 11}, 
$\{P_i\}_{i=0}^N$ are in general position in $\P^N\big(\C\big)$ and $\{Q_j\}_{j=1}^q$ are located in general position in $\mathcal S\big(\{P_i\}_{i=0}^N\big).$ This means that  the hyperplanes $\{H_j\}_{j=1}^q$ are  located in general position in $\P^N\big(\C\big).$ Since $f$ intersects $Q_j$ with multiplicity at least $m_j$ and
$$Q_j(\tilde f)= \big(\sum_{i=0}^Nb_{ji}P_i\big)(\tilde f)= 
\sum_{i=0}^Nb_{ji}\big(P_i(\tilde f)\big)\,,$$
 $F$ also intersects $H_j$ with multiplicity at least $m_j  \ (1\leq j \leq q).$
By Lemma \ref{lemma 10}, $F$ is a constant map, and by Lemma \ref{lemma 12}, $f$ is a  constant map, too.
\end{proof}

\begin{lem}\label{lemma 5nn}
{\it Let natural numbers $N$ and $q \geqslant N+1$ be fixed.
Let $T_i^{(p)}\ (0 \leqslant i \leqslant N, \ p \geqslant 1)$ be moving hypersurfaces in $\P^N\big(\C\big)$ of common degree $d^{(p)}$ and 
 $H_j^{(p)} \in \widetilde{\mathcal{S}}\big(\{T_i^{(p)}\}_{i=0}^N\big)\; (1 \leqslant j \leqslant q, \ p \geqslant 1)$
such that the following conditions are satisfied:

i) For each $0 \leqslant i \leqslant N,$ the coefficients of the homogeneous polynomials $P_i^{(p)}$ which define the $T_i^{(p)}$  are bounded above uniformly  for all $p \geq 1$ on compact subsets of $D$, and for all 
$1 \leqslant j \leqslant q,$
the coefficients $b_{ij}^{(p)}(z)$ of the linear combinations of the 
$P_i^{(p)}$, $i=0,...,N$ which define the
homogeneous polynomials $Q_j^{(p)} = \sum_{i=0}^N b_{ij}^{(p)}P_i^{(p)}$
 which define the $H_j^{(p)}$  are bounded above uniformly  for all $p \geq 1$ on compact subsets of $D$.

ii) There exists $z_0 \in D$  such that 
$$inf_{p \in \N}\big\{D(Q_1^{(p)},...,Q_q^{(p)})(z_0) \big\} > 0 \mbox \,.$$
Then, we have:

a) There exists a subsequence $\{j_p\}\subset \N$ such that for 
$0 \leqslant i \leqslant N$, $P_i^{(j_p)}$ converge uniformly on compact subsets of $D$ to not identically zero homogenous polynomials
$P_i$ (meaning that the $P_i^{(j_p)}$ and $P_i$ are homogenous polynomials in
$ \widetilde{\mathcal H}_D[\omega_0, \cdots, \omega_N]$ of the same degree $d$, and all their coefficients converge uniformly on compact subsets of $D$),
and the $b_{ij}^{(j_p)}$ convergent uniformly on compact subsets of $D$
to $b_{ij} \in {\mathcal H}_D$ for all $0 \leq i \leq N, \: 1\leq j\leq q$.

b) The $Q_j^{(j_p)}= \sum_{i=0}^N b_{ij}^{(j_p)}P_i^{(j_p)}$ converge, for all  $0 \leq i \leq N, \: 1\leq j\leq q$, uniformly on compact subsets of $D$ to
 $Q_j := \sum_{i=0}^N b_{ij}P_i \in  \widetilde{\mathcal{S}}\big(\{P_i\}_{i=0}^N\big) $, and we have
$$ D\big(Q_1, \cdots, Q_q\big)(z_0)  > 0.$$
 In particular
 the moving hypersurfaces
 $Q_1(z_0), \cdots, Q_{q}(z_0)$ are located in general position,
 and the moving hypersurfaces $Q_1(z),...,Q_q(z)$ are located in (weakly) general position.}
\end{lem}

\begin{proof}
Since by our conditions on the coefficients of the $P^{(p)}_i$ and on the $b_{ij}^{(p)}$, for all $1 \leq j \leq q$ the coefficients of the homogenous polynomials $Q_j^{(p)}$ of degree $d^{(p)}$ are locally bounded uniformly for all $p \geq 1$  on compact subsets of $D$, all conditions of
Lemma \ref{lemma 5n} are satisfied  and 
we get that after passing to a subsequence (which we denote
for simplicity again by $\{p\}\subset \N$),
that for 
$1 \leqslant j \leqslant q$, $Q_j^{(p)}$ converge uniformly on compact subsets of $D$ to not identically vanishing homogenous polynomials 
$Q_j$ (meaning that the $Q_j^{(p)}$ and $Q_j$ are homogenous polynomials in
$ \widetilde{\mathcal H}_D[\omega_0, \cdots, \omega_N]$ of the same degree $d_j$, and all their coefficients converge uniformly on compact subsets of $D$). 
Moreover  (still by Lemma \ref{lemma 5n}) we have that
$$D\big(Q_1, \cdots, Q_{q}\big)(z_0)   > 0\,,$$
so the hypersurfaces  $Q_1(z_0), \cdots, Q_{q}(z_0)$ are located in general position,
 and the moving hypersurfaces $Q_1(z),...,Q_q(z)$ are located in (weakly) general position.
 Observe moreover that since all the
$Q_j^{(p)}, \: 1 \leq j \leq q$ were of the same degree $d^{(p)}$, we have
$d=d_j$ independant of $j$ for our subsequence. Hence, we have,
for all $0\leq i \leq N, p \geq 1$:
$$P_i^{(p)}(z)(\omega) = \sum_{| I | = d} a_{ipI}(z).\omega^I\,.$$
Now, since the $a_{jpI}(z) $ and the $b_{ij}^{(p)}(z)$ are locally bounded uniformly for all $p \geq 1$ on $D$,
by using Montel's theorem and a standard diagonal argument with respect to an exaustion of $D$ with compact subsets, 
after passing another time to a subsequence  (which we denote
for simplicity again by $\{p\}\subset \N$),
we also can assume that $\{ a_{ipI}(z)\}_{p=1}^\infty$ converges uniformly on compact subsets of $D$ to $ a_{iI}$ for each $i, I$, and 
that $\{ b_{ij}^{(p)}(z)\}_{p=1}^\infty$ converges uniformly on compact subsets of $D$ to $ b_{ij}(z)$ for each $i, j$.
  Denote by
\begin{equation*}
 P_i(z)(\omega) := \sum_{| I | =d}  a_{iI}(z).\omega^I.
\end{equation*}
Since the limit is unique, then we have  $Q_j= \sum_{i=0}^N b_{ij}P_i$ for $1 \leq j \leq q$ and in
particular that none of the $P_0(z), ..., P_N(z)$ is identically vanishing
(otherwise they could not be in (weakly) general position, which contradicted to
the general position of the $Q_1(z_0),...,Q_q(z_0)$: in fact, if 
the $P_i(z_0)(\omega)$
had a non-zero solution $\omega_0$ in common, so would the $Q_j(z_0)(\omega))$. Hence,  $Q_j  \in  \widetilde{\mathcal{S}}\big(\{P_i\}_{i=0}^N\big) $,
which completes the proof.
\end{proof}

\section{Proofs of the Theorems.}

\begin{proof}[{\bf Proofs of Theorem \ref{theorem 1} and Theorem \ref{theorem 2}}]\
Let $\{f^{(p)}\}$ be a sequence of meromorphic mappings in  $\mathcal F$.
We have to prove that after passing to a subsequence (which we denote
again by $\{f^{(p)}\}$), the sequence  $\{f^{(p)}\}$ converges meromorphically on $D$ to a meromorphic mapping $f$. Moreover, 
under the stronger conditions of Theorem  \ref{theorem 2}, we have to
show that  $\{f^{(p)}\}$ converges uniformly on compact subsets of $D$
to a holomorphic mapping $f$.

In order to simplify notation, we denote, for $1 \leqslant k \leqslant q$,
$$Q_k^{(p)}:=Q_k(f^{(p)}) \; {\rm and}\;  H_k^{(p)}:=H_k(f^{(p)})\,.$$
By Lemma \ref{lemma 5n}, after passing to a subsequence, 
 for all
$1 \leqslant k \leqslant q$, $Q_k^{(p)}$ converge uniformly on compact subsets of $D$ to
$Q_k$, meaning that the 
$$Q_k^{(p)} = Q_k^{(p)}(z)(\omega) = \sum_{| I | = d_k}  a_{kpI}(z).\omega^I \; {\rm and}\; Q_k = 
Q_k(z)(\omega) = \sum_{| I | = d_k}  a_{kI}(z).\omega^I $$
 are homogenous polynomials in
$ \widetilde{\mathcal H}_D[\omega_0, \cdots, \omega_N]$ of the same degree $d_k$, and all their coefficients $a_{kpI}$ converge uniformly on compact subsets of $D$ to $a_{kI}$. Moreover, $Q_1,..., Q_{q}$ are located in
(weakly) general position.

By condition ii) and Lemmas \ref{lemma 2}, \ref{lemma 1}, 
and by condition iii) and Lemmas \ref{lemma 2a}, \ref{lemma 1a}, 
after passing  to a subsequence, we may assume that the sequence $\{f^{(p)}\}$ satisfies
\begin{equation*}
\lim_{p \rightarrow \infty} (f^{(p)})^{-1}\big(H_k^{(p)}\big) = S_k \; (1 \leqslant k \leqslant N+1)
\end{equation*}
as a sequence of closed subsets of $D$, where $S_k$ are either empty or pure $(n-1)$-dimensional analytic sets in $D$, and
\begin{equation*}
\lim_{p \rightarrow \infty} (f^{(p)})^{-1}\big(H_k^{(p)}\big) - S = S_k \; (N+2 \leqslant k \leqslant q)
\end{equation*}
as a sequence of closed subsets of $D - S$, where $S_k$ are either empty or pure $(n-1)$-dimensional analytic sets in $D - S$. 

Let $T = (...,t_{kI},...) \  (1\leqslant k \leqslant q, |I| \leq M:=\max \{d_1,...,d_q\})$ be a family of variables. 
Set $\widetilde{Q}_k = \underset{|I|\leq M}{\sum} t_{kI}\omega^{I} \in \Z [T,\omega] \ (1\leq k\leq q).$
For each subset $L \subset \left\{ 1,...,q\right\}$ with $|L| = n + 1$, take $\widetilde{R}_L $ is the resultant of the $\widetilde{Q}_k \ (k \in L). $ Since 
$\left\{ Q_k\right\} _{k \in L}$ are in (weakly) general position, $\widetilde{R}_L(...,a_{kI},...) \not\equiv  0$ (where we put $a_{kI}=0$ for  $|I| \not= d_k$). We set 
\begin{align*}
\widetilde{S} :=\bigg\{ z \in D |\ \widetilde{R}_L(\cdots,a_{kI},\cdots) = 0& \text{ for some } L \subset \{ 1,\cdots,q\}\\
 & \text{ with } |L| = n + 1\bigg\}.
\end{align*} 
Let $E = (\overset{q}{\underset{k=1}{\cup}} S_k \cup \widetilde{S}) - S$. Then $E$ is either empty or a pure $(n-1)$-dimensional analytic set in $D -S.$

Fix any point $z_1$ in $(D - S) - E.$ Choose a relatively compact neighborhood $U_{z_1}$ of $z_1$ in $(D - S) - E$. Then $\{f^{(p)} \big|_{U_{z_1}}\} \subset Hol\big(U_{z_1}, \P^N\big(\C\big)\big)$. We now prove that the family $\{f^{(p)} \big|_{U_{z_1}}\}$ is a holomorphically normal family. 
Indeed, suppose that the family $\{f^{(p)} \big|_{U_{z_1}}\}$ is not holomorphically normal. By Lemma \ref{lemma 9}, there exist a subsequence (again denoted by $\{f^{(p)} \big|_{U_{z_1}}\}_{p=1}^\infty)$  and $P_0 \in U_{z_1}, \{P_p\}_{p=1}^\infty \subset U_{z_1}$ with $P_p \to P_0$, $\{r_p\} \subset (0, +\infty)$ with $r_p \to 0^+$ and 
$\{u_p\} \subset \C^n$, which are unit vectors, such that $g_p(z) := f^{(p)}\big(P_p + r_pu_pz\big)$ converges uniformly on compact subsets of $\C$ to a nonconstant holomorphic map $g$ of  $\C$ into 
$\P^N\big(\C\big).$
Then, there exist admissible representations $g^{(p)} = \big(g_0^{(p)}:\cdots : g_N^{(p)}\big)$ of $g^{(p)}$ and an admissible  representation $g = \big(g_0: \cdots : g_N\big)$ of $g$ such that the $\{g_i^{(p)}\}$ converge uniformly on compact subsets of $\C$ to $g_i$ $(0 \leq i \leq N)$ (observe that an admissible representation of a holomorphic map is automatically  without common zeros). This implies that $Q_k^{(p)}\big(P_p + r_pu_pz\big)\big(g_0^{(p)}(z),...,g_n^{(p)}(z)\big)$ converges uniformly on compact subsets of $\C$ to 
$Q_k\big(P_0\big)\big(g_0(z),...,g_N(z)\big),$ $(1 \leq k \leq q)$.
Thus, by the Hurwitz's theorem,  one of the following two assertions holds:

i) $Q_k(P_0)(g_0(z),...,g_N(z)) = 0$ on $\C$, i.e. $g(\C) \subset H_k(P_0),$

ii) $Q_k(P_0)(g_0(z),...,g_N(z)) \neq 0$ on $\C$, i.e. $g(\C) \cap H_k(P_0) = \emptyset.$\\
Denote by $J$ the set of all indices $k \in \left\{ 1,...,q\right\}$ with $g(\C) \subset H_k(P_0).$  Set $X = \underset{k \in J}{\cap}H_k(P_0)$ if $J \neq \emptyset$ and $X = \P^N(\C)$ if $J = \emptyset$. Since $\C$ is irreducible, there exists an irreducible component $Z$ of $X$ such that $g(\C) \subset Z - (\underset{k \notin J}{\cup}H_k(P_0))$. 
Since $P_0 \in U_{z_1}$, it implies that $\left\{ H_k(P_0) \right\}_{k=1}^q$ are in general position in $\P^N(\C)$, meaning that
\begin{equation} \label{e1}
{\rm for\:all\:} I \subset \{1,...,q\} {\rm \:with\:}
\# I = N+1\,,\: \underset{k \in I}{\cap}H_k(P_0) = \emptyset\,.
\end{equation}
Put $J^c:= \{1,...,q\} \setminus J$ and $m:=\dim_{\C}Z$.
We claim that for the hypersurfaces $\{H_k(P_0)\}_{k \in J^c}$
in $\P^N(\C)$ we have 
\begin{equation} \label{e2}
\# J^c \geq 2m+1\, ; \; 
{\rm for\:all\:} I \subset J^c {\rm \:with\:}
\# I = m+1\,,\: Z \cap (\underset{k \in I}{\cap}H_k(P_0)) = \emptyset\,.
\end{equation}
In fact, if $J = \emptyset$, so $X=\P^N(\C)$, this holds since $q \geq 2N+1$ and by (\ref{e1}). If $J\not= \emptyset$, the key observation is that by (\ref{e1}) and by B\'{e}zout's theorem we have 
$$
{\rm for\:all\:} 1 \leq l \leq q\,,\:
{\rm for\:all\:} I \subset \{1,...,q\} {\rm \:with\:}
\# I = l\,,\: \underset{k \in I}{\cap}H_k(P_0)$$
\begin{equation} \label{e3}
{\rm is \: of \:pure \: dimension\:} \dim_{\C}\underset{k \in I}{\cap}H_k(P_0) = \max \{N-l, -1\}
\end{equation}
(in particular all irreducible components of $\underset{k \in I}{\cap}H_k(P_0)$ are of the same dimension),
where $\dim_{\C}(\emptyset) = -1$. From that we first get
$$m = \dim_{\C}Z = \dim_{\C}\underset{k \in J}{\cap}H_k(P_0) = \max \{N-\# J, -1\}\,.
$$
Since $g(\C) \subset Z$, so $m \geq 0$, we get 
\begin{equation} \label{e4}
\# J = N-m \,.
\end{equation}
Hence, 
$$\# J^c = q - \# J \geq (2N+1) - (N-m) = N+m+1 \geq 2m+1\,.$$
Moreover, if $I \subset J^c$ with $\# I = m+1$, then by (\ref{e4})
$$\# (I \cup J) = (m+1) + (N-m) = N+1$$ and
$$Z \cap \underset{k \in I}{\cap}H_k(P_0) \subset
\underset{k \in J}{\cap}H_k(P_0)  \cap \underset{k \in I}{\cap}H_k(P_0)
= \underset{k \in I \cup J}{\cap}H_k(P_0) = \emptyset\,,$$
where the last equality follows from (\ref{e1}).
This proves (\ref{e2}) in the case $J \not= \emptyset$.
 By (\ref{e2}) and by \cite[Corollary 1]{E} of Eremenko (or by the more general 
 \cite[Theorem 7.3.4]{No-Wi} of Noguchi-Winkelmann), we get that $Z - (\underset{k \not\in J}{\cup}H_k(P_0))$ is complete hyperbolic and hyperbolically imbedded, and hence $g$ is constant. This is a contradiction. 
 
 Thus $\{f^{(p)}\}$ is a holomorphically normal family on $U_{z_1}$. 
By the usual diagonal argument, we can find a subsequence 
(again denoted by $\{f^{(p)}\}$) which converges uniformly on compact subsets of $(D - S) - E$ to a holomorphic mapping $f$ of $(D - S) - E$ into $\P^N\big(\C\big)$. 

By Lemma \ref{lemma 6n} a), $\{f^{(p)}\}$ has a meromorphically convergent subsequence (again denoted by $\{f^{(p)}\}$) on $D-S$ and again by Lemma \ref{lemma 6n} a), $\{f^{(p)}\}$ has a meromorphically convergent subsequence on 
$D$. Then $\mathcal F$ is a meromorphically normal family on $D$.
The proof of Theorem \ref{theorem 1}  is completed.

Under the additional conditions of Theorem \ref{theorem 2} by Lemma \ref{lemma 6n} b), $\{f^{(p)}\}$ has a subsequence which converges uniformly on compact subsets of $D$ to a holomorphic mapping of $D$ to $\P^N\big(\C\big).$ The proof of Theorem \ref{theorem 2}  is completed.
\end{proof}

\begin{proof}[{\bf Proof of Theorem \ref{theorem 4}}]
Suppose that $\mathcal F$ is not normal on $D$. Then, by Lemma \ref{lemma 9}, there exists a subsequence denoted by $\{f^{(p)}\}\subset \mathcal F$ and $z_0 \in D, \{z_p\}_{p=1}^\infty \subset D$ with $z_p\to z_0$, $\{r_p\} \subset (0, +\infty)$ with $r_p\to 0^+$ and $\{u_p\} \subset \C^n$, which are unit vectors, such that $g^{(p)}(\xi) := f^{(p)}\big(z_p + r_p u_p\xi\big)$ converges uniformly on compact subsets of $\C$ to a nonconstant holomorphic map $g$ of  $\C$ into 
$\P^N\big(\C\big).$

By condition i) of the theorem and by Lemma \ref{lemma 5nn} 
there exists a subsequence (which we denote again by $\{p\}\subset \N$) such that for 
$0 \leqslant i \leqslant N$, $P_i^{(p)}:= P_i(f^{(p)})$ converge uniformly on compact subsets of $D$ to
$P_i$,
and the $b_{ij}^{(p)}:=b_{ij}(f^{(p)})$ converge uniformly on compact subsets of $D$
to $b_{ij}$ for all $0 \leq i \leq N, \: 1\leq j\leq q$ and that
the $Q_j^{(p)}:=Q_j(f^{(p)}) =  \sum_{i=0}^N b_{ij}^{(p)}P_i^{(p)}$ converge, for all 
 $0 \leq i \leq N, \: 1\leq j\leq q$, uniformly on compact subsets of $D$ to
 $Q_j := \sum_{i=0}^N b_{ij}P_i \in  \widetilde{\mathcal{S}}\big(\{P_i\}_{i=0}^N\big) $, and that we have, 
 for any fixed $z=z_0 \in D$,
$$ D\big(Q_1, \cdots, Q_q\big)(z) > \delta(z) > 0$$
 (in particular
 the moving hypersurfaces
 $Q_1(z), \cdots, Q_{q}(z)$ are located in (pointwise) general position).
We finally recall that with writing both variables $z \in D$ and 
$\omega \in \P^N\big(\C\big)$, we thus have that 
$$P_i^{(p)}(z)(\omega) \rightarrow P_i(z)(\omega); \; 
Q_j^{(p)}(z)(\omega) \rightarrow Q_j(z)(\omega); \; 
b_{ij}^{(p)}(z) \rightarrow b_{ij}(z)$$
uniformly on compact subsets in the variable $z \in D$.

For any fixed $\xi_0 \in \C$, there exists a ball $B(\xi_0, r_0)$ in $\C$ and an index $i$ such that $g\left(B(\xi_0, r_0)\right) \subset \{\omega \in \P^N\big(\C\big): \omega_i \not = 0 \}$. Without loss of generality we may assume $i = 0.$ Therefore, there exist admissible representations
$$\tilde g^{(p)}(\xi)=(1,g^{(p)}_1(\xi),\cdots,g^{(p)}_N(\xi)) $$ 
$$\tilde g(\xi)=(1,g_1(\xi),\cdots,g_N(\xi)) $$ of $g^{(p)}$ and $g$ on $B(\xi_0, r_0).$

Because of the convergence of $\{g^{(p)} \}$ on $B(\xi_0, r_0)$, $\{g^{(p)}_i \}$ converges uniformly on compact subsets of $B(\xi_0, r_0)$ to $g_i$ for each $1\,\leq i\leq N.$ This implies that $Q_j^{(p)}\big(z_p+ r_pu_p\xi\big)\big(\widetilde{g}^{(p)}(\xi)\big)$ converges uniformly on compact subsets of $\C$ to $Q_j\big(z_0\big)\big(\widetilde{g}(\xi)\big)$ and $P_i^{(p)}\big(z_p + r_pu_p\xi\big)\big(\widetilde{g}^{(p)}(\xi)\big)$ converges uniformly on compact subsets of $\C$ to $P_i\big(z_0\big)\big(\widetilde{g}(\xi)\big).$

By Hurwitz's theorem, there exists a positive integer $N_0$ such that  
$Q_j^{(p)}\big(z_p+ r_pu_p\xi\big)\big(\widetilde{g}^{(p)}(\xi)\big)$  and $Q_j\big(z_0\big)\big(\widetilde{g}(\xi)\big)$ have the same number of zeros with counting multiplicities on $B(\xi_0, r_0)$ for each $p \geqslant N_0.$ Since the map $g^{(p)}$ of $B(\xi_0, r_0)$ into $\P^N\big(\C\big)$ intersects $Q_j^{(p)}$ with multiplicity at least $m_j$, it implies that any zero $\xi$  of $Q_j\big(z_0\big)\big(\widetilde{g}(\xi)\big)$ has multiplicity at least $m_j.$
Hence,  $g$ intersects $Q_j(z_0)$ with multiplicity at least $m_j$ for each $1\leq j \leq q.$ 

Since 
we have that $Q_1, \cdots, Q_q$ are in $\widetilde{\mathcal S}\big(\{P_i\}_{i = 0}^N\big)$ and 
$$D\big(Q_1, \cdots, Q_q\big)(z) > 0\mbox { for any } z \in D,$$
we have in particular that $Q_1(z_0), \cdots, Q_q(z_0)$ are in ${\mathcal S}\big(\{P_i(z_0)\}_{i = 0}^N\big)$ and 
$$D\big(Q_1, \cdots, Q_q\big)(z_0) > 0.$$
Thus, by Lemma \ref{lemma 13}, $g$ is a constant mapping of $\C$ into $\P^N\big(\C\big).$ This is a contradiction.
\end{proof}

\begin{proof}[{\bf Proofs of Theorem \ref{theorem 5} and Theorem \ref{theorem 6}}]\
Let $\{f^{(p)}\}$ be a se\-quence of meromorphic mappings in  $\mathcal F$.
We have to prove that after passing to a subsequence (which we denote
again by $\{f^{(p)}\}$), the sequence  $\{f^{(p)}\}$ converges meromorphically on $D$ to a meromorphic mapping $f$. Moreover, 
under the stronger conditions of Theorem  \ref{theorem 6}, we have to
show that  $\{f^{(p)}\}$ converges uniformly on compact subsets of $D$
to a holomorphic mapping $f$.

By condition i) of the theorems and by Lemma \ref{lemma 5nn} 
there exists a subsequence (which we denote again by  $\{f^{(p)}\}$) such that for 
$0 \leqslant i \leqslant N$, $P_i^{(p)}:= P_i(f^{(p)})$ are homogenous polynomials of the same degree $d$ and converge uniformly on compact subsets of $D$ to
$P_i$,
and the $b_{ij}^{(p)}:=b_{ij}(f^{(p)})$ converge uniformly on compact subsets of $D$
to $b_{ij}$ for all $0 \leq i \leq N, \: 1\leq j\leq q$ and that
the $Q_j^{(p)}:=Q_j(f^{(p)}) =  \sum_{i=0}^N b_{ij}^{(p)}P_i^{(p)}$ converge, for all 
 $0 \leq i \leq N, \: 1\leq j\leq q$, uniformly on compact subsets of $D$ to
 $Q_j := \sum_{i=0}^N b_{ij}P_i \in  \widetilde{\mathcal{S}}\big(\{P_i\}_{i=0}^N\big) $, and that we have
$$ D\big(Q_1, \cdots, Q_q\big)(z_0)  > 0.$$
In particular,  the moving hypersurfaces
 $Q_1(z_0), \cdots, Q_{q}(z_0)$ are located in general position,
 and the moving hypersurfaces $Q_1(z),...,Q_q(z)$ are located in (weakly) general position.

By condition ii) of Theorem \ref{theorem 5} and Lemmas \ref{lemma 2}, \ref{lemma 1}, 
and by condition iii) of the theorems and Lemmas \ref{lemma 2a}, \ref{lemma 1a}, 
after passing  to a subsequence, we may assume that the sequence $\{f^{(p)}\}$ satisfies
\begin{equation*}
\lim_{p \rightarrow \infty} (f^{(p)})^{-1}\big(H_k^{(p)}\big) = S_k \; (1 \leqslant k \leqslant N+1)
\end{equation*}
as a sequence of closed subsets of $D$, where $S_k$ are either empty or pure $(n-1)$-dimensional analytic sets in $D$, and
\begin{equation*}
\lim_{p \rightarrow \infty} 
\overline{\left\{z \in \supp \nu\big(f^{(p)}, H_k^{(p)}\big)\big|\nu\big(f^{(p)}, H_k^{(p)})\big)(z) < m_k \right\}}
- S = S_k \; (N+2 \leqslant k \leqslant q)
\end{equation*}
as a sequence of closed subsets of $D - S$, where $S_k$ are either empty or pure $(n-1)$-dimensional analytic sets in $D - S$. 

Let $T = (...,t_{kI},...) \  (1\leqslant k \leqslant q, |I| =d)$ be a family of variables. 
Set $\widetilde{Q}_k = \underset{|I|=d}{\sum} t_{kI}\omega^{I} \in \Z [T,\omega] \ (1\leq k\leq q).$
For each subset $L \subset \left\{ 1,...,q\right\}$ with $|L| = n + 1$, take $\widetilde{R}_L $ is the resultant of the $\widetilde{Q}_k \ (k \in L). $ Since 
$\left\{ Q_k\right\} _{k \in L}$ are in (weakly) general position, $\widetilde{R}_L(...,a_{kI},...) \not\equiv  0$ (where we put $a_{kI}=0$ for  $|I| \not= d$). We set 
\begin{align*}
\widetilde{S} :=\bigg\{ z \in D |\ \widetilde{R}_L(\cdots,a_{kI},\cdots) = 0& \text{ for some } L \subset \{ 1,\cdots,q\}\\
 & \text{ with } |L| = n + 1\bigg\}.
\end{align*} 
Let $E = (\overset{q}{\underset{k=1}{\cup}} S_k \cup \widetilde{S}) - S$. Then $E$ is either empty or a pure $(n-1)$-dimensional analytic set in $D -S.$

Fix any point $z_1$ in $(D - S) - E.$ Choose a relatively compact neighborhood $U_{z_1}$ of $z_1$ in $(D - S) - E$. Then $\{f^{(p)} \big|_{U_{z_1}}\} \subset Hol\big(U_{z_1}, \P^N\big(\C\big)\big)$. We now prove that the family $\{f^{(p)} \big|_{U_{z_1}}\}$ is a holomorphically normal family. For this it is sufficient to observe that the family $\{f^{(p)} \big|_{U_{z_1}}\}$ now satisfies all conditions of Theorem \ref{theorem 4}:
In fact there exists $N_0$ such that for $p \geq N_0$, 
$\{f^{(p)} \big|_{U_{z_1}}\}$ does not intersect $H_k^{(p)}$ for $1 \leq k\leq N+1$, and $\{f^{(p)} \big|_{U_{z_1}}\}$ intersects $H_k^{(p)}$ of order at least $m_k$ for $N+2 \leq k\leq q$, and for all $z \in U_{z_1}$, we have 
$ D\big(Q_1, \cdots, Q_q\big)(z)  > 0$. So if we still put $m_k= \infty$ for 
$1\leq k \leq N+1$, the conditions of Theorem \ref{theorem 4} are satisfied,
and so the family $\{f^{(p)} \big|_{U_{z_1}}\}$ is a holomorphically normal family.
By the usual diagonal argument, we can find a subsequence 
(again denoted by $\{f^{(p)}\}$) which converges uniformly on compact subsets of $(D - S) - E$ to a holomorphic mapping $f$ of $(D - S) - E$ into $\P^N\big(\C\big)$. 

By Lemma \ref{lemma 6n} a), $\{f^{(p)}\}$ has a meromorphically convergent subsequence (again denoted by $\{f^{(p)}\}$) on $D-S$ and again by Lemma \ref{lemma 6n} a), $\{f^{(p)}\}$ has a meromorphically convergent subsequence on 
$D$. Then $\mathcal F$ is a meromorphically normal family on $D$.
The proof of Theorem \ref{theorem 5}  is completed.

Under the additional conditions of Theorem \ref{theorem 6} by Lemma \ref{lemma 6n} b), $\{f^{(p)}\}$ has a subsequence which converges uniformly on compact subsets of $D$ to a holomorphic mapping of $D$ to $\P^N\big(\C\big).$ The proof of Theorem \ref{theorem 6}  is completed.
\end{proof}

{\bf Acknowledgements.} This work was done during a stay of the authors at 
the Vietnam Institute for Advanced Study in Mathematics (VIASM).
We would like to thank the staff there, in particular the partially support of VIASM.

\vspace{1cm}

{\it Gerd Dethloff$^{1,2}$\\
$^1$ Universit\'e Europ\'eenne de Bretagne, France\\
$^2$ Universit\'e de Brest \\
Laboratoire de Math\'{e}matiques Bretagne Atlantique - UMR CNRS 6205\\ 
6, avenue Le Gorgeu, C.S. 93837\\
29238 Brest Cedex 3, France}\\
email: gerd.dethloff@univ-brest.fr

\vskip0.3cm

{\it Do Duc Thai and Pham Nguyen Thu Trang\\
Department of Mathematics\\
Hanoi National University of Education\\
136 XuanThuy str., Hanoi, Vietnam}\\
emails: ducthai.do@gmail.com; \  pnttrangsp@yahoo.com

\end{document}